\documentclass[a4paper]{amsart}

\usepackage[margin=1.40 in]{geometry}

\usepackage{amsmath, amssymb, amsthm}
\usepackage{mathabx}
\usepackage{pstricks, pst-node}
\usepackage{cite}
\usepackage{graphicx}
\usepackage{pinlabel}
\usepackage{enumerate}
\usepackage{verbatim}
\usepackage{epigraph}

\newcommand{\N}{\mathbb{N}}

\newcommand{\modulus}[1]{\left|#1\right|}
\newcommand{\PCL}{\operatorname{PCL}}

\newcommand{\wh}[1]{\widehat{{#1}}}

\newtheorem{thm}{Theorem}[section]

\newtheorem{lem}[thm]{Lemma}
\newtheorem{prop}[thm]{Proposition}

\newtheorem{thmspecial}{Theorem}
\newtheorem{corspecial}[thmspecial]{Corollary}

\newtheorem{propspecial}[thmspecial]{Proposition}
\theoremstyle{definition}
\newtheorem{defn}[thm]{Definition}

\theoremstyle{remark}

\newtheorem{question}[thm]{Question}

\newcommand{\abs}[1]{\lvert#1\rvert} 

\newcommand{\gen}[1]{\left\langle#1\right\rangle}

\newcommand{\ga }{\Gamma}

\newcommand{\e }{\varepsilon }
\renewcommand{\kappa }{\varkappa}

\renewcommand{\d }{{\rm d} }

\newcommand{\cH }{\mathcal H}

\newcommand{\cQ }{\mathcal Q}
\newcommand{\cP }{\mathcal P}

\newcommand{\grate}{\mathbf{h}}
\newcommand{\cgrate}{\mathbf{k}}

\newcommand{\ball}{\mathbb{B}}

\newcommand{\pcconj}{PC--conjugator}
\newcommand{\CLF}{\operatorname{CLF}}
\newcommand{\qPCL}{\operatorname{PCL}_{G,X,\lambda,c}}
\newcommand{\qPCLp}{\operatorname{PCL}_{G,Y,\lambda',c'}}

\title[Permute and Conjugate]{Permute and conjugate: the conjugacy problem in relatively hyperbolic groups}
\address{Vanderbilt University,
	Department of Mathematics,
	1326 Stevenson Center,
	Nashville, TN 37240, USA.}
\author{Yago Antol\'{i}n}
\email[Yago Antol\'{i}n]{yago.anpi@gmail.com}

\author{Andrew Sale}
\email[Andrew Sale]{andrew.sale@some.oxon.org}

\date{\today}

\begin{document}

	\begin{abstract}
		Modelled on efficient algorithms for solving the conjugacy problem in hyperbolic groups, we define
		and study the permutation conjugacy length function. This function estimates the length
		of a short conjugator between words $u$ and $v$, up to taking cyclic permutations.		
		This function might be bounded by a constant, even in the case when the standard conjugacy length
		function is unbounded. We give  applications to the complexity of the conjugacy problem, estimating conjugacy growth rates, and languages. Our main result states that for a relatively hyperbolic group, the permutation conjugacy length function is bounded by the permutation conjugacy length function of the parabolic subgroups.
		\\ 		 
		\flushright{We came, we permuted, we conjugated.}
	\end{abstract}
	
\keywords{relatively hyperbolic groups, Cayley graphs, conjugacy growth, conjugacy problem, conjugacy length function}

\subjclass[2010]{20F65, 20F10, 20F67}
	\maketitle
		
		\vspace{-10mm}
			\section{Introduction}\label{sec:intro}

	Max Dehn's decision problems have long been among the most fundamental problems in combinatorial and geometric group theory. Especially with the development of geometric techniques, we have seen a surge in progress in our understanding of the word and conjugacy problems over recent years.
	
	We develop a geometric approach to solving the conjugacy problem that runs in a similar vein to that of the conjugacy length function, but has the potential to provide a computationally faster algorithm than the na\"ive algorithm associated to the conjugacy length function. It also has applications to computing the conjugacy growth rate and deciding the regularity of conjugacy languages.

	We let $G$ be group with finite generating set $X$. 
	We begin by defining the conjugacy length function for $G$. Suppose that $u,v$ are words on $X$ which represent conjugate elements of $G$. An element $w$ of $G$ is a \emph{conjugator} for $u,v$ if $wu=vw$.
	
	\begin{defn}
		Define the \emph{conjugacy length function} of $G$ to be the function
		$$\CLF_{G,X}:\N \to \N$$
		which takes $n\in \N$ to the following value:
		$$\max\limits_{\ell(u)+\ell(v) \leq n} \Big\{ \min\{ \ell(w) \mid \textrm{ $w$ is a conjugator for $u,v$}\}  \Big\}.$$
	\end{defn}
	
	The conjugacy length function is discussed in certain classes of groups in \cite{Saleext,Salelie,Salewreath}. See the introduction of these papers for further references.
	
	We remark that the conjugacy length function is stable when changing the generating set, up to the usual asymptotic equivalence of functions as set out in Definition \ref{def:asymequiv} below. We may therefore omit the generating set from the notation and just write $\CLF_G$ for the conjugacy length function of $G$.
	
	\begin{defn}\label{def:asymequiv}We consider  functions up to
		the following equivalence: we say $f \preceq g$ if there exists a constant $C$ such that
		$f(n)\leq  Cg(Cn)+C$  for all $n \in \mathbb{N}$, and we say $f \asymp g$ if both $f \preceq g$ and $g \preceq f$.
	\end{defn}
	
	We now move on to define a variation on $\CLF_G$, which we call the permutation conjugacy length function (PCL) of $G$. A key difference between this and the conjugacy length function is that PCL is defined on words, whereas the conjugacy length function may be defined using elements.
	
	Whenever we consider a word $w$ on $X$, we will use the notation $w$ to denote both the word and the element in $G$ it represents.  In case we need to emphasize that two words are equal, we use $\equiv$ to denote letter-by-letter equality. 
	
	Given a word $u \equiv x_1\ldots x_r$ on $X$, a \emph{cyclic permutation} of $u$ is a word of the form $u' \equiv x_i\ldots x_rx_1\ldots x_{i-1}$, for any $i=1,\ldots , r$.
	
	Suppose $u$ and $v$ are words on $X$ which represent conjugate elements of $G$. An element $w$ is a \emph{permutation conjugacy conjugator}, or more briefly a \emph{\pcconj}, for $u$ and $v$ if there exist cyclic permutations $u'$ and $v'$ of $u$ and $v$ respectively such that $wu'=v'w$.
	
	\begin{defn}\label{def:pcl}
		Let $u,v$ be words on $X$ representing conjugate elements of $G$. Define  $\PCL_{G,X}(u,v)$ to be the length of the shortest conjugator between all cyclic permutations of the  words $u$ and $v$. That is:
		$$\PCL_{G,X}(u,v) := \min\{\ell(w) \mid w \textrm{ is a \pcconj{} for $u$ and $v$}\}.$$
		The \emph{permutation conjugacy length function} of $G$ is the function
		$$\PCL_{G,X}(n):=\max\{\PCL_{G,X}(u,v) \,\mid \, u,v \text{ geodesics satisfying }\ell(u)+\ell(v)\leq n\}.$$
	\end{defn}
	
	We discuss below, in Section \ref{sec:change gen set}, its behaviour when changing generating set.
	
	The permutation conjugacy length function is closely related to the standard conjugacy length function. Indeed, we have
	\begin{equation}\label{eq: PCL and CLF}
	\PCL_{G,X}(n) \leq \CLF_{G,X}(n) \leq \PCL_{G,X}(n) + \frac{n}{2}.\end{equation}
	The first inequality is clear from the definitions, while the latter is realised as follows: if $u,v$ have a \pcconj{} $w$ such that $wu'=v'w$, then we obtain a conjugator $w_0$ for $u,v$ which acts first by cyclically permuting $u$ to $u'$, then conjugating by $w$ to get $v'$, then cyclically permuting to $v$.

	A consequence of these inequalities is that if we know the conjugacy length function to be super-linear, then the same will also apply to the permutation conjugacy length function. Therefore, studying the permutation conjugacy length function is only of interest when the group in question has (sub)linear conjugacy length function. 
	The simplest example of a group $G$ satisfying  $$\PCL_{G,X}\not\asymp \CLF_{G,X}\asymp n,$$ is a non-abelian free group. Proposition \ref{propspecial:pcl hyp}   shows this  can be extended to the family of  hyperbolic groups.

	In certain situations, studying the permutation conjugacy length function may give a better understanding of conjugacy in a group than one gets from the conjugacy length function. For example, it may lead to a  faster algorithm solving the conjugacy problem than one may achieve by na\"ively using $\CLF_G$. Indeed, notions similar to the PCL are used in recent algorithms giving fast solutions in hyporbolic and relatively hyperbolic groups \cite{Bumagin14,BridsonHowie,EpsteinHolt,BuckleyHolt}. These methods stemmed from a lemma of Bridson and Haefliger \cite[Chapter III.$\Gamma$ Lemma 2.11]{BH} which asserts that, for hyperbolic groups, PCL is bounded by a constant when restricting to words $u,v$, all of whose cyclic permutations are local geodesics. We remove this restriction by studying PCL for all (quasi-)geodesic words.
	
	The main result of this paper is the following.
	
	\begin{thmspecial}\label{thmspecial:pcl}
		Let $G$ be a finitely generated group, hyperbolic relative to subgroups $\{H_\omega\}_{\omega \in \Omega}$. Then there exists a finite generating set $X$ such that $\langle X\cap H_\omega\rangle =H_\omega$ for $\omega \in \Omega$ and
		$$\PCL_{G,X}(n) \leq 2\max\limits_{\omega\in\Omega}\{\PCL_{H_\omega,X\cap H_\omega}(n)\} + K$$
		for some constant $K$. 
	\end{thmspecial}

	We remark that the lower bound
		$$\max\limits_{\omega\in\Omega}\{\PCL_{H_\omega,X\cap H_\omega}(n)\} \preceq \PCL_{G,X}(n)$$
	follows from the definitions and the fact that the parabolic subgroups $H_\omega$ are almost malnormal in $G$ and isometrically embedded.
		
		We note that  $\CLF_G$ is a constant function if and only if the group is finite-by-abelian (see \cite{BFC}). The class of groups with constant $\PCL$ is significantly bigger, including groups hyperbolic relative to finite-by-abelian subgroups, as Theorem \ref{thmspecial:pcl} shows. We remark that there is a virtually abelian group, proposed by Holt, which we can show has linear $\PCL$.

	An important family of groups that are relatively hyperbolic groups are limit groups, where the parabolic subgroups $H_\omega$ are abelian (see \cite{Dahmani}). Theorem \ref{thmspecial:pcl} therefore gives a constant upper bound for the PCL of these groups.
		
	The complexity of the conjugacy problem in relatively hyperbolic groups has already been well-studied, including work of Bumagin \cite{Bumagin14} and O'Conner \cite{Zoe}. The conjugacy length function, when it is polynomial in the subgroups $H_\omega$, is investigated by Ji, Ogle and Ramsey \cite{JiOgleRamsey}.
	
	An immediate consequence of Theorem \ref{thmspecial:pcl}, together with  observation  \eqref{eq: PCL and CLF} above, is a generalisation of the result of Ji, Ogle and Ramsey \cite{JiOgleRamsey} concerning the conjugacy length function in relatively hyperbolic groups.
	
	\begin{corspecial}
		Let $G$ be a finitely generated group, hyperbolic relative to subgroups $\{H_\omega\}_{\omega \in \Omega}$. Then, either $\CLF_G(n)$ is sublinear or
				$$\CLF_G(n) \asymp  \max\limits_{\omega \in \Omega} \{\CLF_{H_\omega}(n)\} + n.$$
	\end{corspecial}
	
	In the preceding result, we remark that $\CLF_G$ may be sublinear only in certain cases when the relative hyperbolic structure is degenerate.
	
	The geometric techniques we use to determine a bound on the permutation conjugacy length function for relatively hyperbolic groups give the following:
	
	\begin{propspecial}\label{propspecial:pcl hyp}
		Let $G$ be a hyperbolic group. Then for every finite generating set $X$ there exists a constant $C=C(X,\delta)$ such that
				$$\PCL_{G,X}(n) \leq C.$$
	\end{propspecial}

	We note also that Proposition \ref{propspecial:pcl hyp} is a generalisation of the aforementioned lemma of Bridson and Haefliger \cite{BH}, extending the constant bound to all geodesic words. When restricting to cyclic geodesics, this property is also known as the ``bounded conjugacy diagram'' property (BCD) in \cite{AC14}, where it was studied for relatively hyperbolic groups.
		
	Proposition \ref{propspecial:pcl hyp} may of course be viewed as a corollary of Theorem \ref{thmspecial:pcl}, however this would be a rather convoluted way to prove it. Indeed, in Section \ref{sec:hyp cayley graphs} below, we give a short proof of this fact, en route to proving the result for relatively hyperbolic groups.

	The structure of the paper is as follows. In Section \ref{sec:pcl} we  discuss some applications of the permutation conjugacy length function, namely complexity, conjugacy growth, and regular languages. We also investigate its behaviour under a change of generating set.
	
	In Section \ref{sec:hyp cayley graphs} we set-up what may be viewed as a general strategy for dealing with PCL in groups with a hyperbolic Cayley graph.
		Section \ref{sec:rel hyp} is devoted to defining relatively hyperbolic groups and collecting useful results from the literature.
	Finally, Section \ref{sec:proof} gives the proof of Theorem \ref{thmspecial:pcl}.

	As mentioned before, $\PCL$ is defined on words. It is worth mentioning here that throughout this paper, we assume that generating sets generate  groups as  monoids. In particular, we don't require generating sets to be closed under inversion.

\medskip
{\bf Acknowledgements:} The authors wish to thank M.{} Bestvina, L.\ Ciobanu, T.\ Riley, and R.{} Wade for various discussions during the writing of this paper, and to the latter also for comments on a draft version of this paper.

\section{The permutation conjugacy length function and its applications}\label{sec:pcl}

\subsection{Complexity of the conjugacy problem}

If $\CLF_G$ is bounded above by a computable function and $G$ has solvable word problem, then one can solve the conjugacy problem in $G$.
Indeed, given elements $u$ and $v$, with $n=\modulus{u}_X+\modulus{v}_X$, one just needs to check if an element in $\ball_X(\CLF_G(n))$,
the ball of radius $\CLF_G(n)$, conjugates $u$ to $v$.
Note that, even in the case when $\CLF_G(n)$ is bounded by a linear function, this na\"ive algorithm is far from being efficient and will run in exponential time, with respect to $n$, for groups of exponential growth.

A potentially important application of the permutation conjugacy length function is that it can make the previous  algorithm more efficient. When PCL is bounded by a constant, this algorithm will be almost as fast as the word problem.

\begin{prop}\label{prop:complexity}
	Suppose $G$ is generated by a finite set $X$ and that there is an algorithm solving the word problem in time $O(f(n))$. 
	Then there is an algorithm which takes as input two geodesic words $u,v$ and determines whether they represent conjugate elements of $G$, and if they do it will give a conjugator, in time
			$$O\bigg(n^2\modulus{\ball_X(\PCL(n))}f\big(n+2\PCL(n)\big)\bigg)$$
	where $n= \ell(u)+\ell(v)$ and $\PCL(n) = \PCL_{G,X}(n)$.
	
	In particular, if $\PCL(n) \leq A$ then the algorithm runs in time
			$$O\big(n^2f(n+2A)\big).$$

\end{prop}

\begin{proof}
We use the algorithm solving the word problem to check	all $n^2$ combinations of cyclic permutations of $u$ and $v$, in each case checking all elements in the radius $\PCL(n)$ ball in $G$ to see if it is a \pcconj{} for $u,v$. Each check is an application of an algorithm solving the word problem for an input word of length at most $n+2\PCL(n)$.
\end{proof}

	Observe that if $\PCL_{G,X}(n)$ is bounded by a logarithm and if $f$ is polynomial, then Proposition \ref{prop:complexity} gives a polynomial-time algorithm.

\subsection{Application to conjugacy growth rate}

The {\it growth function of $G$ with respect to $X$} is the function 
 $$\gamma_{G,X} (n) := \left|\ball_X(n)\right|.$$ Let $\sim_c$ denote the $G$--conjugacy equivalence relation, that is $u \sim_c v$ if and only if $u$ and $v$ are conjugate elements in $G$. The {\it conjugacy growth function of $G$ with respect to $X$} is the function 
 $$\xi_{G,X}(n) := \left|\ball_X(n)/\sim_c\right|$$
that counts the number of $G$--conjugacy classes in the ball of radius $n$. Up to the asymptotic equivalence of functions (see Definition~\ref{def:asymequiv}),  $\gamma_{G,X}$ and $\xi_{G,X}$ are independent of the chosen generating set. We may therefore omit the generating set and just write $\gamma_G$ or $\xi_G$ where clear.

In \cite{GubaSapir}, Guba and Sapir conjectured, and verified in many instances, that finitely presented groups with exponential growth functions also have exponential conjugacy growth functions.
Similar to our discussion about the complexity of the conjugacy problem, a priori, even if $\CLF_G$ is linear we cannot extract information about $\xi_G$ from  $\CLF_G$ and $\gamma_G$.  On the other hand, if $\PCL_{G,X}$ is sublinear then we can conclude that $\xi_G$ is exponential whenever $\gamma_G$ is too. Moreover, we can give bounds on the growth rates. Recall that the {\it exponential growth rate} of a group $(G,X)$ is equal to
$$\grate_{G,X}=\limsup_{n\rightarrow \infty} \log \sqrt[n]{\gamma_{G,X}(n)},$$
 and the {\it exponential conjugacy growth rate} is equal to $$\cgrate_{G,X}=\limsup_{n\rightarrow \infty} \log \sqrt[n]{\xi_{G,X}(n)}.$$
 
\begin{prop}\label{lem:conj growth}
Let $X$ be a finite generating set of $G$.
Then 
\begin{equation}\label{eq:growth}
\dfrac{\gamma_{G,X}(n)}{{n\gamma_{G,X}\big(\PCL_{G,X}(2n)+n/2\big) }}  
\leq \xi_{G,X}(n)\leq \gamma_{G,X}(n).\end{equation}
In particular, if $\PCL_{G,X}(n)=o(n)$ then  $\dfrac{1}{2}\grate_{G,X} \leq \cgrate_{G,X}\leq \grate_{G,X}$.
\end{prop} \begin{proof}
The right-hand inequality of \eqref{eq:growth} is trivial. So we just need to prove the left-hand inequality.
For that, fix $u$ a geodesic word over $X$ of length less than or equal to $n$. We will show that $[u]_c(n)$, the intersection of the conjugacy class of $u$ with $\ball_X(n)$, has size
$$\left|[u]_c(n)\right|\leq {n\left|\ball_X(\PCL_{G,X}(2n)+n/2)\right|} .$$
Indeed, let $v\in [u]_c(n)$. Then there is a cyclic permutation $u'$ of $u$ and a cyclic permutation $v'$ of $v$ 
such that $u'$ and $v'$ are conjugated by an element of length $\PCL_{G,X}(2n)$. Note that $v$ is conjugated to $v'$ by 
an element of length $n/2$. Then every element of $[u]_c(n)$ is conjugated to a cyclic permutation of $u$ by an element of 
length at most $\PCL_{G,X}(2n)+n/2$. This gives the upper bound on the size of $[u]_c(n)$, which in turn proves the left-hand inequality of \eqref{eq:growth}.

Recall that $\gamma_{G,X}(n)$ is submultiplicative (i.e. $\gamma_{G,X}(p+q)\leq \gamma_{G,X}(p)\gamma_{G,X}(q)$). Hence by Fekete's Lemma (see
\cite[VI.56]{delaHarpe}) $\lim_{n\to \infty}\sqrt[n]{\left|\ball_X(n)\right|}$  exists. Let this limit be $\lambda\geq 1$.
Suppose that $\PCL_{G,X}(n)=o(n)$. Then for every $\e>0$ there exists $N>0$ such that $\PCL_{G,X}(2n)<\e n$ for all $n\geq N$.
From $\eqref{eq:growth}$, we have for $n\geq N$ that 
$$ \dfrac{\modulus{\ball_X(n)}}{n \left|\ball_X(\e n)\right| \left|\ball_X(n/2)\right|} {\leq \frac{|\modulus{\ball_X(n)}}{n\left|\ball_X(\PCL_{G,X}(2n)+n/2)\right|}} \leq  \xi_{G,X}(n).$$ Then we have that $$\lim_{n\to \infty}\sqrt[n]{\dfrac{\left|\ball_X(n)\right|}{n\left|\ball_X(\e n)\right|  \left|\ball_X(\frac{n}{2})\right|}} = \dfrac{\lambda}{\lambda^\e \lambda ^{1/2}}=\lambda^{\frac{1}{2}-\e}.$$
Hence for every $\e >0$, we have $\left(\frac{1}{2}-\e\right)\grate_{G,X}\leq\cgrate_{G,X}$.
\end{proof}

We remark that as a consequence of Proposition~\ref{lem:conj growth}, we get that $\frac{1}{2}\grate_{G,X} \leq \cgrate_{G,X}\leq \grate_{G,X}$ for any group $G$ that is hyperbolic relative to finite-by-abelian subgroups.

\subsection{Application to regular languages}

The language of conjugacy geodesics was introduced in \cite{CH} by Ciobanu and Hermiller, and it has been further studied in \cite{AC14,CHHR}. Fix $X$ to be a finite generating set of $G$, a  word $w$ is said to be a {\it conjugacy geodesic} if its
length is minimal among representatives of elements in its $G$--conjugacy class. The set of conjugacy geodesics is denoted by $\mathsf{ConjGeo}(G,X)$. When $\PCL$ is a constant function, then it might be possible to establish that $\mathsf{ConjGeo}(G,X)$ is a regular language. The following is implicit in the proof of \cite[Theorem 3.1]{CHHR} and explicit in \cite[Corollary 3.8]{CHHR}
\begin{prop}
Suppose that $\PCL_{G,X}(n)$ is constant.
\begin{enumerate} 
\item[{\rm (1)}] If there is a biautomatic structure on $\mathsf{Geo}(G,X)$, then $\mathsf{ConjGeo}(G,X)$ is regular.
\item[{\rm (2)}] If $(G,X)$ has the falsification by fellow traveler property, then $\mathsf{ConjGeo}(G,X)$ is regular.
\end{enumerate}
\end{prop}

\subsection{Behaviour under a change of generating set}\label{sec:change gen set}

We first note that changing the generating set will leave $\PCL$ invariant whenever $\CLF_G$ is super-linear. This follows from observation \eqref{eq: PCL and CLF} in Section \ref{sec:intro} and the invariance of $\CLF$ under changing generating sets. 

Let $X$ and $Y$ be two finite generating sets for a group $G$. Consider an element $g\in G$ which is represented by a word $u$ on $X$ and $v$ on $Y$. The set of elements in $G$ given by the cyclic permutations of $u$ may be very different to those from cyclic permutations of $v$. Therefore, in general, it is not clear how the permutation conjugacy length function will behave when changing generating set.

To relate words in $X$ to words in $Y$, we rewrite each element of $X$ as a word on $Y$, and use this to rewrite $u$ to give a word $u_0$ on $Y$ representing $g$. If $u$ was a geodesic word, then $u_0$ will be a quasi-geodesic word, or more generally if $u$ was a quasi-geodesic word then $u_0$ will also be quasi-geodesic, but with larger constants.

Under this process however, the set of elements represented by cyclic permutations of $u$ is contained in the set of elements represented by cyclic permutations of $u_0$. This motivates the following definition.

\begin{defn}
	The permutation conjugacy length function for $(\lambda,c)$--quasi-geodesic words, with respect to the generating set $X$, is the function
	$\qPCL:\N\to\N$
	which takes $n\in \N$ to
	$$\max\big\{\PCL(u,v) \,\mid \, u,v \text{ are $(\lambda,c)$--quasi-geodesics satisfying }\ell(u)+\ell(v)\leq n\big\}.$$
\end{defn}

If we understand the permutation conjugacy length function for quasi-geodesics, then we may change the generating set and maintain control over the length of \pcconj s, as described in the following:

\begin{lem}
	Suppose $X$ and $Y$ are finite generating sets for $G$. For each $\lambda\geq 1$, $c\geq 0$ there exists $\lambda' \geq \lambda$, $c'\geq c$ and $K>0$ such that
	$$\qPCL(n) \leq  K^2 + K\qPCLp(Kn),\ \ \textrm{ for all $n \in \N$}.$$
\end{lem}

\begin{proof}
	For each $x \in X$, let $w_x$ be a word on $Y$ representing the same element in $G$ as $x$. Similarly for $y \in Y$ let $w_y$ be a word on $X$ representing the same element as $y$. We define
	$$K:=\max\{\ell_X(w_y),\ell_Y(w_x) \mid x\in X, y\in Y\}.$$
	
	Let $u,v$ be words on $X$ that are $(\lambda,c)$--quasi-geodesics representing conjugate elements in $G$. Convert them into words on $Y$ by exchanging each generator with the corresponding word $w_y$ to obtain words $\hat{u},\hat{v}$ on $Y$ representing the same elements. Then there will be $\lambda'\geq \lambda$ and $c'\geq c$ such that $\hat{u},\hat{v}$ are $(\lambda',c')$--quasi-geodesic words.
	
	We leave as an exercise to the reader to prove
	$$\PCL_{G,X}(u,v) \leq K^2+K\PCL_{G,Y}(\hat{u},\hat{v}).$$
	The Lemma follows immediately.
\end{proof}

We ask for groups where the (geodesic) permutation conjugacy length function does depend on the generating set. In particular:
	\begin{question} 
		Can one find a group $G$ with finite generating sets $X$ and $Y$ such that $\PCL_{G,X}$ is constant but $\PCL_{G,Y}$ is unbounded.
	\end{question}

\section{Hyperbolic Cayley graphs}\label{sec:hyp cayley graphs}\label{sec:hyp cayley graph}

In this section we consider the case where $G$ has some generating set $S$, which may be infinite, with respect to which the Cayley graph is $\delta$--hyperbolic.

Suppose $u,v,w$ are geodesic words on $S$ such that $wu=vw$. From this we can construct a geodesic quadrilateral in the Cayley graph of $G$. If $w$ is instead a \pcconj{} for $u,v$, then the polygon we want to construct will be a hexagon. 
Indeed, we consider the hexagon $\cQ $ in the Cayley graph $\Gamma = \ga (G,S)$, as in Figure \ref{fig:hexagon}, which has geodesic sides. It will have two opposite sides labelled by $w$. If $wu'=v'w$, then the six vertices of $\cQ $ will be at $1,w,wu_2,$ $wu',v',v_2$, where $u\equiv u_1u_2$, $u'\equiv u_2u_1$ and $v\equiv v_1v_2$, $v'\equiv v_2v_1$.

	\begin{figure}[h!]
		\labellist
		\small \hair 2pt
		\pinlabel ${u_2}$ at 40 510
		\pinlabel ${u_1}$ at 210 510
		\pinlabel ${v_2}$ at 55 25
		\pinlabel ${v_1}$ at 230 68
		\pinlabel $w$ at -7 285
		\pinlabel $w$ at 248 285
		\pinlabel $1$ at -15 165
		\pinlabel $wu'=v'w$ at 320 410
		\pinlabel $v'$ at 267 166
		\endlabellist
		\includegraphics[height=7cm]{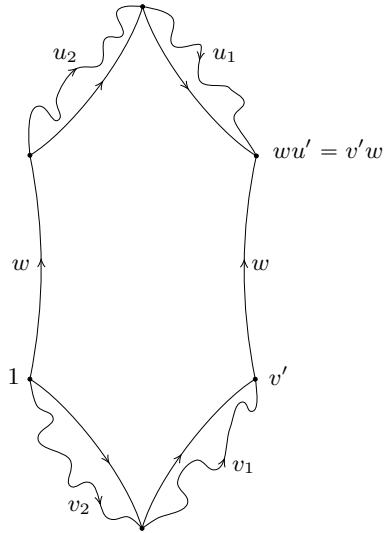}
		\caption{The quasi-geodesic hexagon $\cQ $ recognising $w$ as a \pcconj{} of $u\equiv u_1u_2$ and $v\equiv v_1v_2$. When the Cayley graph is $\delta$--hyperbolic the quasi-geodesic edges will be within a $B$--neighbourhood of a geodesic, for a constant $B=B(\delta,\lambda,c)$.}\label{fig:hexagon}
	\end{figure}
	
In general we will take $u,v$ to be quasi-geodesics, meaning the four sides of $\cQ $ labelled by subwords of $u$ or $v$ will now be quasi-geodesics rather than geodesics.

The following result asserts that in a hyperbolic Cayley graph, the hexagon $\cQ $ will be skinny in a very precise sense.

\begin{lem}\label{lem:quasi-pcl on hyperbolic Cayley graph}
	Suppose a group $G$ has a (possibly infinite) generating set $S$ such that the Cayley graph $\ga (G,S)$ is $\delta$--hyperbolic for some $\delta$.
	Let $\lambda\geq 1,c\geq 0$ and suppose $u,v$ are  $(\lambda,c)$--quasi-geodesic words on $S$ representing conjugate elements of $G$. 
	
	Then there exists a constant $B=B(\delta,\lambda,c)$ such that if $w$ is a geodesic word on $S$ representing a \pcconj{} for $u,v$ of minimal length, and if $w_i$ is the length $i$ prefix of $w$, then for $4\delta+B < i < \abs{w}_S-4\delta-B$,
			$$d_S(w_i,v'w_i) \leq 8 \delta$$
	where $v'$ is the cyclic permutation of $v$ such that $wu'=v'w$ for some cyclic permutation $u'$ of $u$.
\end{lem}

\begin{proof}
	Let $u$ and $v$ be the quasi-geodesic words and $w\equiv s_1s_2\ldots s_n$ a word on $S$ of minimal length such that $wu'=v'w$ for some cyclic permutations $u',v'$ of $u,v$ respectively. Let $\cQ $ be the hexagon in Figure \ref{fig:hexagon}, defined above.
	
	We use the property that a point on a side  of a geodesic hexagon in a hyperbolic space lies within distance $4\delta$ from a point on one of the other sides. Since four of the sides of our hexagon $\cQ $ are $(\lambda,c)$--quasi-geodesics, we need to add a constant $B=B(\delta,\lambda,c)$ for each of these sides involved. Here $B$ is chosen so it satisfies the property that any $(\lambda,c)$--quasi-geodesic lies in the $B$--neighbourhood of some geodesic connecting its endpoints.

	Let $w_i \equiv s_1\ldots s_i$ be the length $i$ prefix of $w$ and consider the vertex on the left edge of the hexagon corresponding to the group element  $w_i$. Suppose $4\delta+B < i < \abs{w}_S-4\delta-B$. Then there is some point on one of the other edges which is close to $w_i$. If it is on the opposite side, then it will be within a distance of $4\delta$ (since both sides of $\cQ $ involved are geodesics). If it is on any of the other sides, then it will be within distance $4\delta+B$, since that side is a quasi-geodesic.
	
	The key point to observe here is that given the assumption on $i$, if $w_i$ is within distance $4\delta+B$ of some vertex in a side labelled by a subword of $u$ or $v$, then we are able to take a short-cut and obtain a \pcconj{} with shorter $S$--length. See Figure \ref{fig:hexagon2} (a). 

	\begin{figure}[h!]
		\labellist
			\small \hair 2pt
			\pinlabel $\textrm{(a)}$ at 133 590
			\pinlabel $u_2$ at 72 486
			\pinlabel $u_1$ at 192 486
			\pinlabel $w_i$ at 50 270
			\pinlabel $v'w_i$ at 205 270
			\pinlabel $>4\delta+B$ at -50 225
			\pinlabel $<4\delta+B$ at 105 210
			\small
			\pinlabel $v_2$ at 78 68
			\pinlabel $v_1$ at 188 68
			\pinlabel $\textrm{(b)}$ at 620 590
			\pinlabel $u_2$ at 438 486
			\pinlabel $u_1$ at 561 486
			\pinlabel $w_i$ at 375 270
			\pinlabel $v'w_i$ at 580 280
			\pinlabel $x$ at 630 350
			\pinlabel $v'x$ at 880 350
			\pinlabel $<4\delta$ at 500 328
			\pinlabel $>4\delta$ at 647 310
			\pinlabel $v_2$ at 442 68
			\pinlabel $v_1$ at 556 68
		\endlabellist
			\vspace{0.5cm}
		\includegraphics[height=8cm]{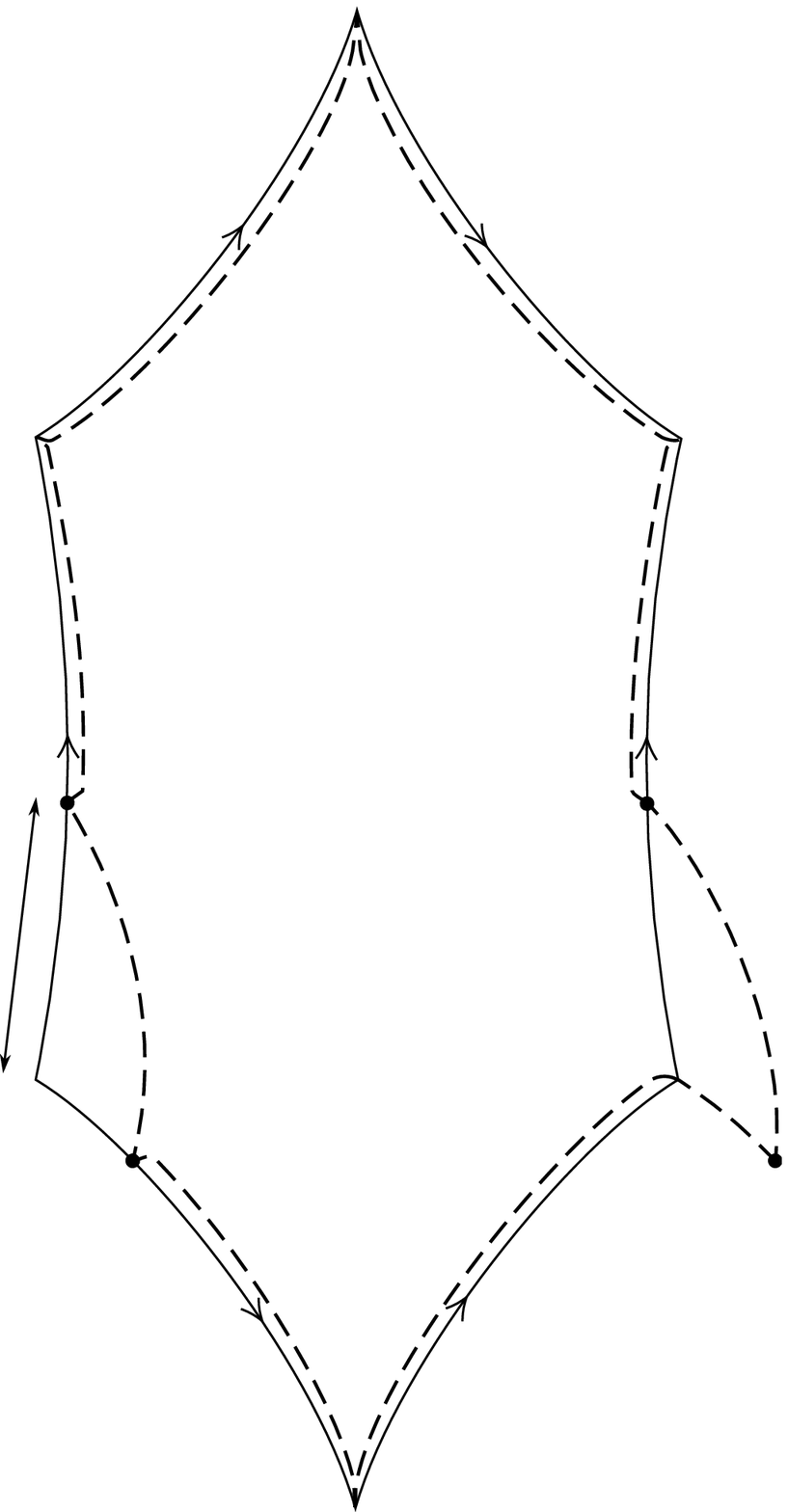} \hspace{1cm} \includegraphics[height=8cm]{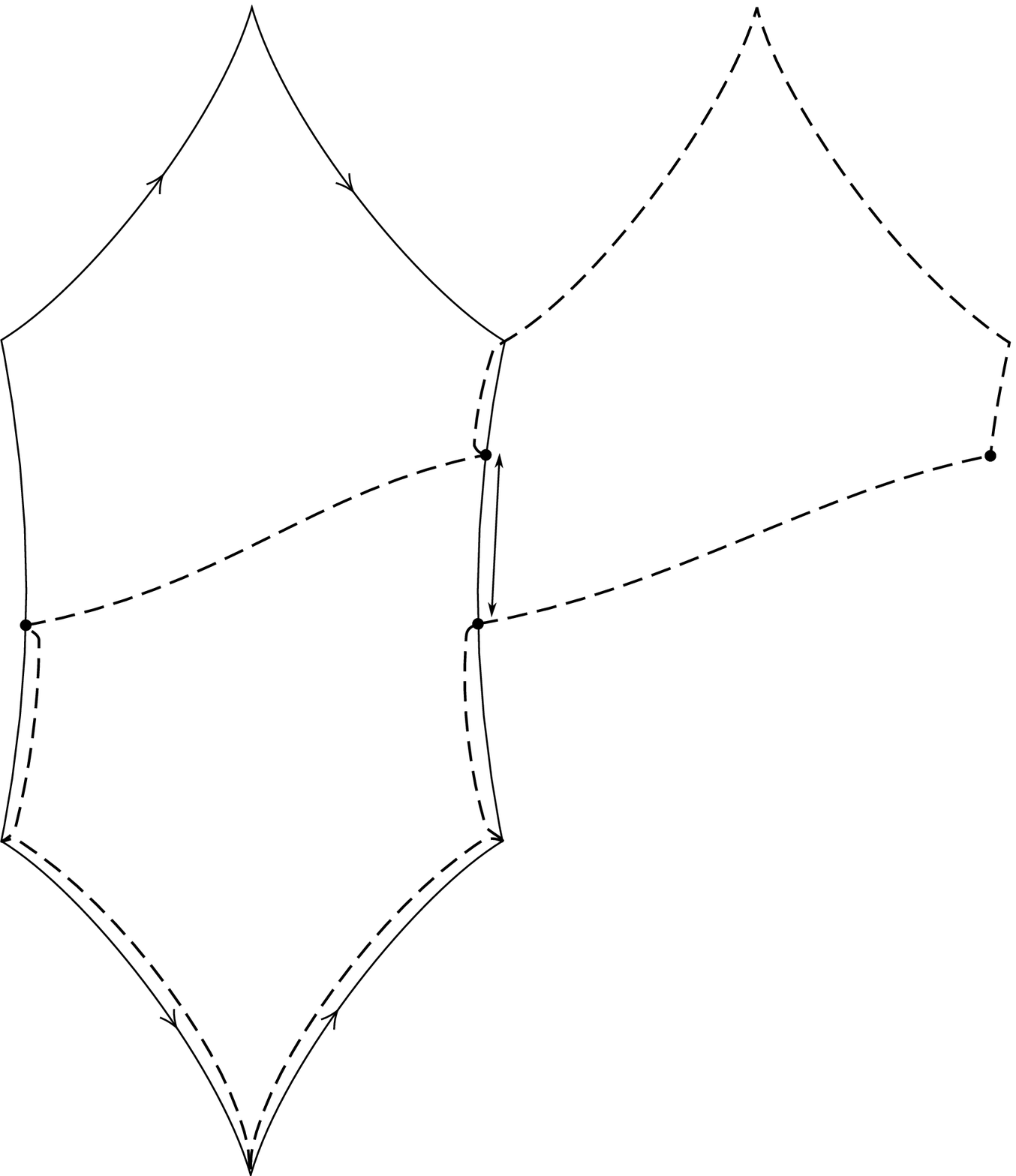} 
		\caption{Short-cuts giving a word $w'$ on $S$ which is a \pcconj{} for $u=u_1u_2,v=v_1v_2$ but $\abs{w'}_S<\abs{w}_S$. The dotted lines outline the hexagons representing the new conjugacy diagrams.}\label{fig:hexagon2}
	\end{figure}

	Hence the point close to $w_i$ on another side of $\cQ$ must in fact lie on the opposite side, i.{}e.{} the other side labelled by $w$, and hence be within $4\delta$. Let $x$ be this point. If this point is further than $4\delta$ from the vertex $v'w_i$ then we will be able to take a short-cut across the hexagon from $x$ to $w_i$, see Figure \ref{fig:hexagon2} (b), enabling us to find a \pcconj{} of shorter $S$--length. Hence
			$$d_S(w_i,v'w_i) \leq d_S(w_i,x)+d_S(x,v'w_i) < 8\delta,$$
	proving the Lemma.
\end{proof}

We now deal with the special case where the group is hyperbolic, showing that $\PCL_{G,X}(n)$ is bounded by a constant.

\begin{proof}[Proof of Proposition \ref{propspecial:pcl hyp}]
	We use Lemma \ref{lem:quasi-pcl on hyperbolic Cayley graph} with $S=X$. We further assume that $u,v$ are geodesic words so, by the definition of $B$ in the proof of Lemma \ref{lem:quasi-pcl on hyperbolic Cayley graph}, we have $B=0$.
	Let $w,w_i,u',v'$ be as in the proof of Lemma \ref{lem:quasi-pcl on hyperbolic Cayley graph}.
	
	We can consider the elements represented by $w_i^{-1}v'w_i$ for $4\delta<i<\modulus{W}_X-4\delta$ and observe that if we have repetition among these elements then we can cut a middle section from the diagram and create a new, shorter conjugator. Indeed, if $i<j$ and $w_i^{-1}v'w_i=w_j^{-1}v'w_j$, then if we set $w_0 = w_iw_j^{-1}w$ one can verify that $w_0u'=v'w_0$. However $\modulus{w_0}_X=\modulus{w}_X-(j-i)$. Hence if $w$ is of minimal length, then we can have no repetition, giving
		$$\abs{w}_S \leq \modulus{\ball_X(8\delta)} + 8\delta$$
and completing the proof.
\end{proof}

As we mentioned, $\PCL$ is inspired by \cite[Chapter III.$\Gamma$ Lemma 2.11]{BH}
which says that if $\ga(G,S)$ is hyperbolic  (without assumptions on finiteness of $S$)
then there is a $K>0$ such that for any pair of cyclic geodesics $u,v$,
if $\max\{\ell(u),\ell(v)\}>K$ then $\PCL_{G,X}(u,v)\leq K$.

We want to point out two differences.   Since \cite[Chapter III.$\Gamma$ Lemma 2.11]{BH} deals with geodesic quadrilaterals,
once one finds a \pcconj{} of minimal length, the geodesics labelled by the cyclic permutations of $u$ and $v$ synchronously $K_1$-fellow travel, where $K_1$ is a constant
depending on $\delta$ and $K$. In particular, one can see that it is only necessary to cyclically permute one of the cyclic geodesic words to get
a conjugator of length less than $K_1$. This is not the case when one consider geodesics. This is a key observation that
allows one to solve the conjugacy problem in hyperbolic groups in linear time, compared to Proposition \ref{prop:complexity}
that gives a quadratic bound.

The second difference arises in the bound of the length of the conjugator. In \cite[Chapter III.$\Gamma$ Lemma 2.11]{BH}
it is not necessary to have a finite generating set, let us explain why. Using the assumption that $u$ and $v$ are long enough cyclic geodesics,
one can show that no pair of vertices in different  $w$--sides in Figure \ref{fig:hexagon} are at distance less that $8\delta$,
in this case Lemma \ref{lem:quasi-pcl on hyperbolic Cayley graph} gives that the conjugator has to have $S$--length at most
$8\delta +2B$. If one considers general geodesics, either some finiteness condition (like in Proposition \ref{propspecial:pcl hyp}) or a lower bound on translation length
becomes crucial in order to bound the conjugator.

\section{Relatively hyperbolic groups}\label{sec:rel hyp}

In this section we collect some preliminary material concerning relatively  hyperbolic groups.

Let $G$ be a group, $X$ a finite generating set, and $\{H_\omega\}_{\omega\in\Omega}$ a collection of subgroups. Let $\Gamma = \Gamma(G,S)$ be the Cayley graph of $G$ with respect to the generating set $S$, where
			$$S=X \cup \bigcup_{\omega \in \Omega} H_\omega \setminus \{1\}.$$
Given an edge path $p$ in $\ga$ we denote by $p_{-}$ and $p_+$ the initial and terminal vertices of $p$.

\begin{defn}
	Let $p$ be a path in $\ga$.
	\begin{enumerate}
		\item An {\it $H_\omega$--component} of ${p}$ is a subpath ${s}$ such that all the edges of $s$ have labels from $H_\omega$  and $s$ is not contained in any other subpath of $p$ with this property. A subpath ${s}$ is a {\it component} if it is an $H_\omega$--component for some $\omega\in \Omega$. For a component $s$ of $p$ and a generating set $Y$ of $G$, the  {\it $Y$--length of $s$} is  $\d_Y(s_-,s_+)$.
		
		\item Two components ${s}$ and ${r}$ (not necessarily in the same path) are {\it connected} if both are $H_\omega$--components for
		some $\omega\in \Omega$ and all the vertices of $s$ and all the vertices of $r$ lie in the same coset of $H_\omega$.  
		\item A component ${s}$ of ${p}$ is {\it isolated} if it is not connected to any other component ${r}$ of ${p}$.  
	\end{enumerate}
\end{defn}
	
		\begin{defn}
			A group $G$ is \emph{hyperbolic relative to a collection of subgroups} $\{H_\omega\}_{\omega \in \Omega}$ if:
	\begin{enumerate}
		\item the Cayley graph $\Gamma(G,S)$ is $\delta$--hyperbolic for some $\delta\geq 0$, and
		\item the pair $(G,\{H_\omega\}_{\omega\in \Omega})$ satisfies the \emph{bounded coset penetration property}:
		for any $\lambda\geq 1$ there exists $a=a(\lambda)>0$ such that for all $(\lambda,0)$--quasi-geodesics $p,q$ in $\Gamma(G,S)$ with $p_-=q_-$ and $d_X(p_+,q_+)\leq 1$, then
		\begin{enumerate}
			\item if a subpath $s$ of $p$ is an $H_\omega$-component and $d_X(s_-,s_+)\geq a$, then $s$ is connected to an $H_\omega$--component of $q$,
			\item if $s,t$ are connected components of $p,q$ respectively, then $d_X(s_-,t_-)\leq a$ and $d_X(s_+,t_+)\leq a$.
		\end{enumerate}
	\end{enumerate}
	\end{defn}
	By \cite[Lemma 6.9, Theorem 6.10]{Osin06} this is equivalent to the versions of relative hyperbolicity given by Bowditch \cite{Bowditch}, Farb \cite{Farb}, and Osin \cite{Osin06}.

In the following three lemmas, $G$ denotes a group with finite generating set $X$, and which is hyperbolic relative to a collection of subgroups $\{H_\omega\}_{\omega\in \Omega}$.
We first need the following result, which is a version of \cite[Proposition 3.2]{OsinPF}.
\begin{lem}\label{lem:n-gon}	
	There exists $D=D(G,X,\lambda,c)>0$ such that the following holds. Let 
	$\mathcal{P}=p_1p_2\cdots p_n$ be an $n$--gon in $\ga(G,X\cup \cH)$ and  $I$ a distinguished subset  of sides of $\mathcal{P}$ such that if $p_i\in I$, $p_i$ is an isolated component in $\mathcal{P}$, and
	if $p_i\notin I$, $p_i$ is a $(\lambda,c)$--quasi-geodesic. Then
	$$\sum_{i\in I}\d_X((p_i)_-,(p_i)_+)\leq D n.$$
\end{lem}

If $G$ is finitely generated, then $\Omega$ is finite (see \cite[Theorem 1.1]{Osin06}) and  $H_\mu \cap H_\omega$ is finite, for $\omega \neq \mu$ (see \cite[Theorem 1.4]{Osin06}). It follows that
\begin{lem}
\label{lem:nocancellation}
There exists $m$ such that if $h \in H_\mu \cap H_\omega$, for $\omega \neq \mu$, then $\abs{h}_X \leq m$.
\end{lem}

Given a word $u\equiv x_1x_2\ldots x_n$ on $X$, we can rewrite this as a (potentially) shorter word $\wh{u}$ on $S$  by replacing, from left to right, maximal subwords of $u$ whose letters are all contained in the same subgroup $H_\omega$ with the element of $H_\omega$ that subword represents. (In the language of \cite{AC14}, $\wh{u}$ is derived from $u$.)

The following  is a specific case  of \cite[Lemma 5.3]{AC14}.
\begin{lem}\label{lem:gen set lemma}
	There exists $\lambda\geq 1$, $c\geq 0$ and a finite generating set $X$ for $G$ such that
	
	\begin{enumerate}
	\item[{\rm (1)}] $\gen{X\cap H_\omega}=H_\omega$ for all $\omega \in \Omega$, 
	\item[{\rm (2)}] the natural embedding $\ga(H_\omega,X\cap H_\omega)$ into $\ga(G,X)$ is an isometric embedding for all $\omega\in \Omega$,
	\item[{\rm (3)}] if $u$ is a geodesic word over $X$ representing some element of $H_\omega$, then $u$ is a word over $H_\omega \cap X$,
	\item[{\rm (4)}] if $u$ is a geodesic word on $X$ then $\wh{u}$ will be a $(\lambda,c)$--quasi-geodesic in $\ga $ and no subword $\wh{u}_1$ of $\wh{u}$ with $\ell(\wh{u}_1)>1$ represents an element of $\cH$.
	\end{enumerate}
\end{lem}

\section{Bounded PCL for relatively hyperbolic groups}\label{sec:proof}

Let $G$ be hyperbolic relative to subgroups $H_\omega$ for $\omega \in \Omega$. Let $X$ be the finite generating set of $G$ from Lemma \ref{lem:gen set lemma}, and as above let $\Gamma$ be the Cayley graph of $G$ with respect to the generating set 
	$$S=X \cup \bigcup_{\omega \in \Omega} H_\omega \setminus \{1\}.$$
Let $\delta$ be the hyperbolicity constant for $\Gamma$. We denote by $d_X$ and $d_S$ the word metrics induced by the generating sets $X$ and $S$ respectively, and by $\modulus{\cdot}_X$ and $\abs{\cdot}_S$ the corresponding word lengths.

We will often use the phrases $X$--length or $S$--length to refer to the length of elements under the corresponding generating set.

Let $u,v$ be geodesic words on $X$ which represent conjugate elements of $G$. Let $w$ be a \pcconj{} for $u$ and $v$ for which $\modulus{w}_S$ is minimal.
Let $s_1\ldots s_n$ be a geodesic word on $S $ representing $w$, so $n=\modulus{w}_S$.  Let $w_i$ be the length $i$ prefix of $w$.

Throughout we will refer to a hexagon $\cQ $ in $\Gamma$ which will have quasi-geodesic sides, defined as 
follows.  We remark that $\cQ $ is similar to the hexagon of Section \ref{sec:hyp cayley graphs}, but for the appropriate words in $S$. See Figure \ref{fig:hexagon}.
 It will have vertices at $1,w,wu_2,$ $wu',v',v_2$, where $u\equiv u_1u_2$ and $u'\equiv u_2u_1$, and similarly for $v$, so $wu'=v'w$. Along the opposite sides from $1$ to $w$ and $v'$ to $v'w$ we read $s_1\ldots s_n$. Along the other four sides we read the $(\lambda,c)$--quasi-geodesic words $\wh{u_1},\wh{u_2},\wh{v_1}$ and $\wh{v_2}$, which are obtained from $u_1,u_2,v_1,v_2$ respectively as in Lemma \ref{lem:gen set lemma}.

Notice that if every edge in one of the two sides of the hexagon $\cQ$ labelled by $w$ is isolated in $\cQ$ then we have
		$$\modulus{w}_X\leq 8D\modulus{w}_S$$
where $D$ is the constant of Lemma \ref{lem:n-gon}. Indeed, we take each edge individually and treat it as the eighth side of a geodesic $8$--gon. Lemma \ref{lem:n-gon} gives us that the $X$--length of each $s_i$ is bounded above by $8D$.

Thus it will be helpful, and indeed necessary, to obtain a bound on the $S$--length of $w$. Note that we cannot use methods of Proposition \ref{propspecial:pcl hyp} since we lack local finiteness. Also we see that the interesting cases are when some edge is not isolated. In light of this, we first collect some lemmas to help deal with the non-isolated edges. A frequent tactic that we appeal to is that of finding a shorter path from a side of $\cQ $ labelled by a subword of $u$ to a side labelled by a subword of $v$. We call this {\it taking a short-cut}. This process  gives a new \pcconj{} of shorter $S$--length.

If the $i$--th edge of the path from $1$ to $w$ is connected to the $i$--th edge in the path from $v'$ to $v'w$, then we call this a \emph{horizontal band}.
If the first or last edge of either side corresponding to $w$ is connected to the adjacent edge on the neighbouring $u$ or $v$--side of the hexagon $\cQ$ then we call this a \emph{corner chunk}. See Figure \ref{fig:corner_and_horizontal}.

\begin{figure}[h!]
	\labellist
	\small \hair 2pt
		\pinlabel {corner chunk} at -50 400
		\pinlabel {horizontal band} at -70 230
		\pinlabel $w_i$ at 42 270
		\pinlabel $v'w_i$ at 310 270
		\pinlabel $v'w_{i-1}$ at 325 225
	\endlabellist
	\includegraphics[height=8cm]{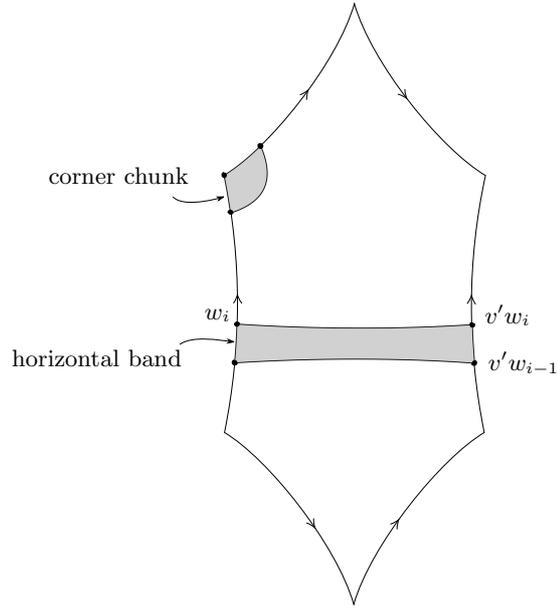}
	\caption{The shaded regions represent edges in the same coset of a parabolic subgroup.}\label{fig:corner_and_horizontal}
\end{figure}

\begin{lem}\label{lem:bands}
If $w\equiv s_1s_2\dots s_n$ is a \pcconj{} of minimal $S$--length 
	and the edge labelled by $s_i$ in the side from $1$ to $w$ is not isolated, then:
	\begin{enumerate} 
		\item[{\rm (1)}] if $1<i<n$ then it is connected via a horizontal band to the opposite $w$--side of $\cQ $;
		\item[{\rm (2)}] if $i=1,n$ and it is connected to the opposite $w$--side, then it is connected via a horizontal band.
	\end{enumerate}
\end{lem}
\begin{proof}
Suppose it is connected to some other edge of $\cQ $. Figure \ref{fig:X-skinny-middle} shows how we can take a short-cut and find a \pcconj{} of shorter $S$--length. 
\end{proof}

\begin{figure}[h!]
	
	\labellist
	\small \hair 2pt
	\pinlabel $\textrm{(a)}$ at 120 600
	\pinlabel $w_i$ [r] at 12.3 263.7
	\pinlabel $v'w_i$ [r] at 228.8 263.7
	\pinlabel $\wh{v_2}$ at 67 68
	\pinlabel $\wh{v_1}$ at 177 68
	\pinlabel $\wh{u_2}$ at 60 488
	\pinlabel $\wh{u_1}$ at 180 488
	
	\pinlabel $\textrm{(b)}$ at 570 600
	\pinlabel $w_{i-1}$ at 315 255
	\pinlabel $v'w_j$ at 590 347
	\pinlabel $\wh{v_2}$ at 396 68
	\pinlabel $\wh{v_1}$ at 506 68
	\pinlabel $\wh{u_2}$ at 389 488
	\pinlabel $\wh{u_1}$ at 509 488

	\endlabellist
	
	\vspace{8mm}
	\includegraphics[height=8cm]{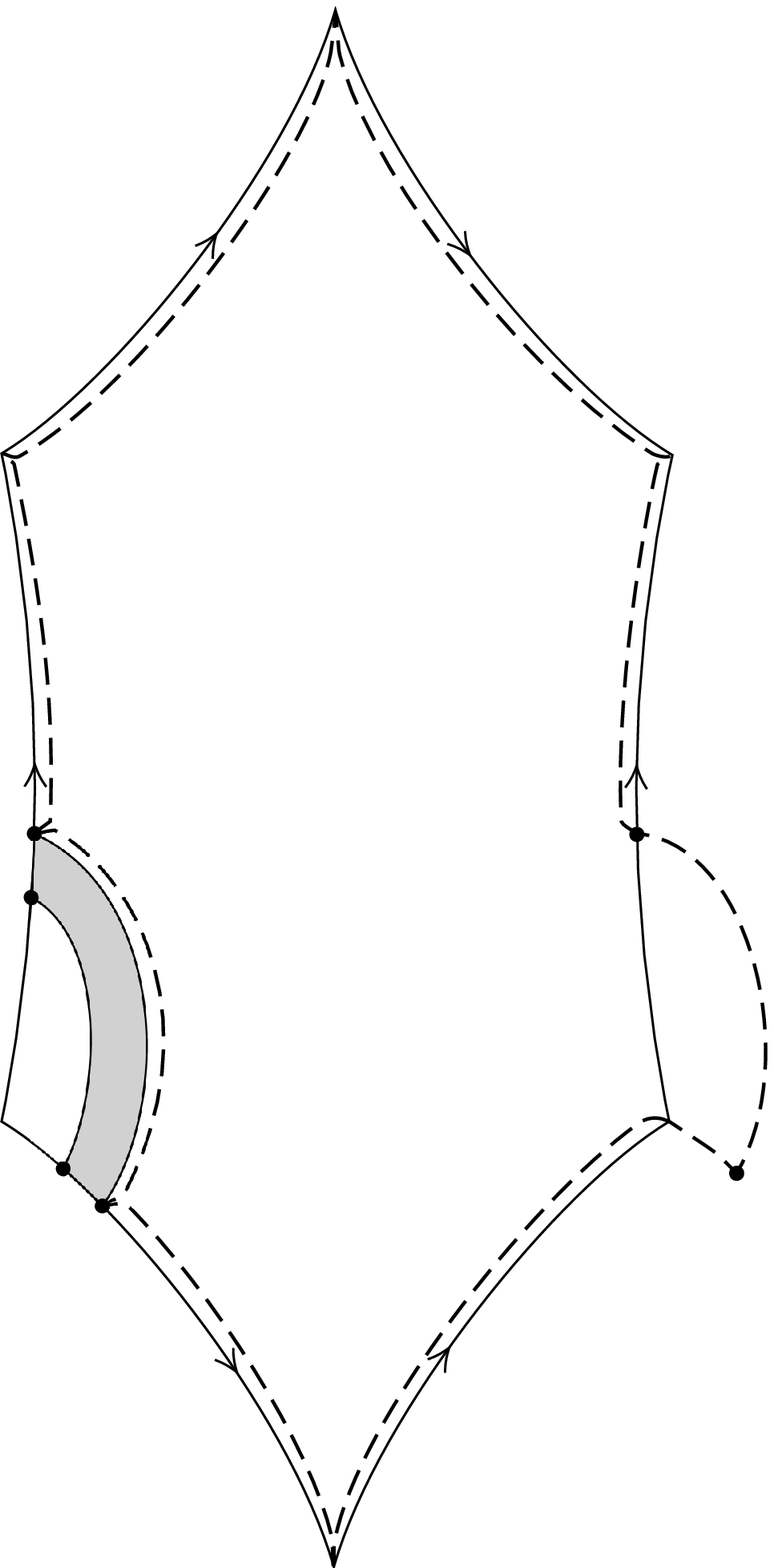} \hspace{5mm}
	\includegraphics[height=8cm]{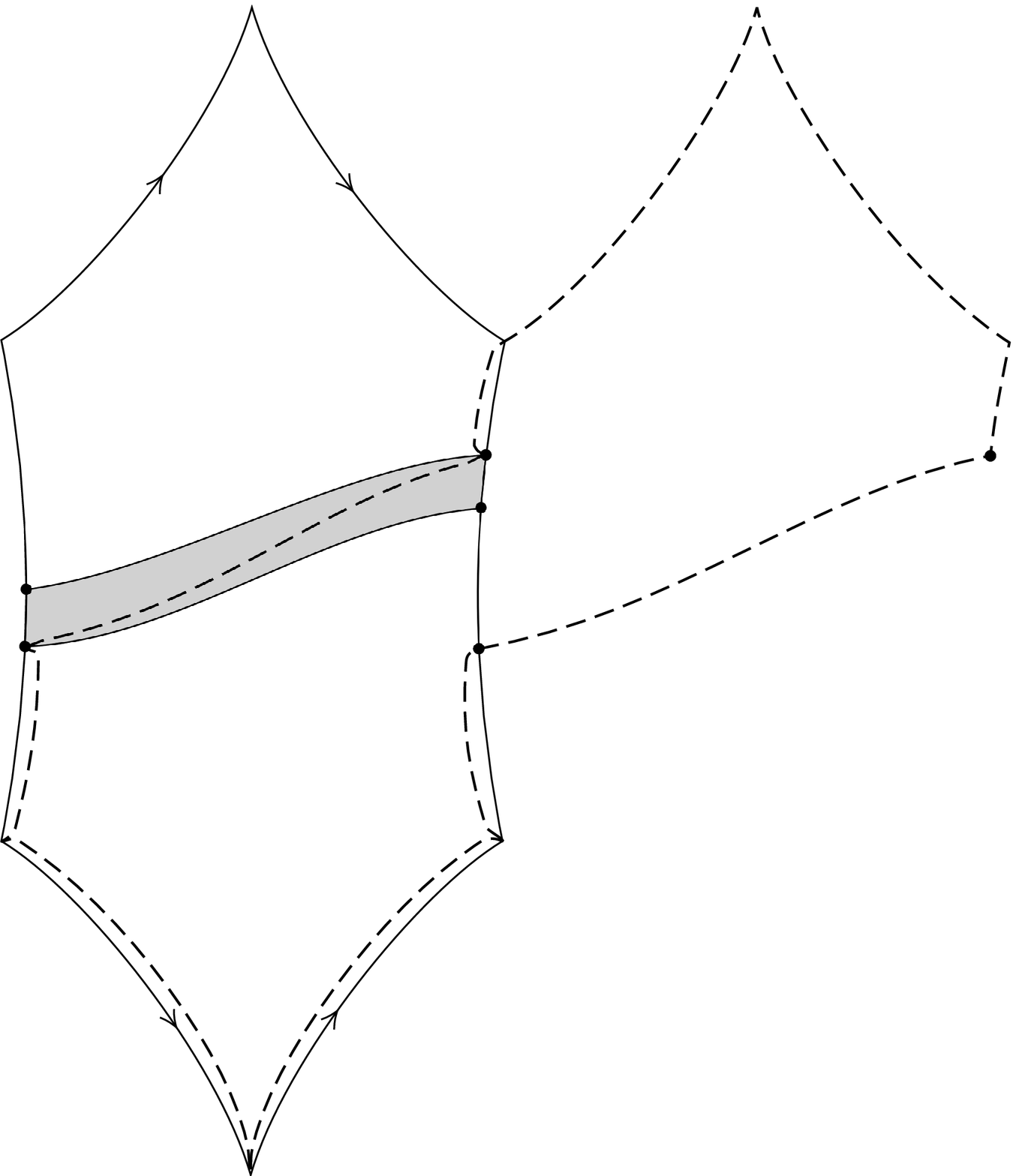} 
	\caption{The shaded regions represent  edges with labels in the same $H_\omega$. Note that the four vertices in a shaded region span a complete subgraph in $\ga$.  (a) A short-cut is obtained by following the dotted path and obtaining a \pcconj{} of shorter $S$--length. (b) Since $d_S(w_{i-1},v'w_j)=1$, the dotted path represents a $S$--shorter cyclic conjugator whenever $|i-j|\geq 1$. 
	}\label{fig:X-skinny-middle}
\end{figure}

\begin{lem}\label{lem:bounding_middle_conj}
	If $w\equiv s_1s_2\dots s_n$ is a \pcconj{} of minimal $S$--length then there exists a constant $M=M(G,X)$ such that for each $1<i<n$ we can assume that
	$$\modulus{s_i}_X \leq M.$$
\end{lem}

\begin{proof} 	
	If the edge labelled by $s_i$ in the side from $1$ to $w$ is isolated, then we are done provided $8D\leq M$, where $D$ is the constant of Lemma \ref{lem:n-gon}. Hence we assume it is not isolated, so by Lemma \ref{lem:bands} it is connected to the opposite $w$--side via a horizontal band. 
	
	Let $g$ be the label of the edge connecting $w_i$ to $v'w_i$. This edge is isolated in the pentagon with vertices at $w_i,w,v'w,v'w_i$, or else we can take a short-cut and obtain an $S$--shorter \pcconj{}, see Figure \ref{fig:more short cuts}. Hence $\modulus{g}_X\leq 5D$ by Lemma \ref{lem:n-gon}. If the edge from $w_{i-1}$ to $v'w_{i-1}$ has label $h$, then we can apply similar reasoning to get $\modulus{h}_X\leq 5D$.
	
	\begin{figure}[h!]
		\labellist
		\small \hair 2pt
			\pinlabel $g$ at 130 267
			\pinlabel $x$ at 44 446
			\pinlabel $\geq 2$ at -20 318
			\pinlabel $w_{i-1}$ at -10 230
			\pinlabel $w$ at 0 410
			\pinlabel $g$ at 560 267
			\pinlabel $x$ at 428 375
			\pinlabel $v'x$ at 695 380
			\pinlabel $\geq 2$ at 410 300
			\pinlabel $w_{i-1}$ at 420 230
		\endlabellist
		
		\includegraphics[height=8cm]{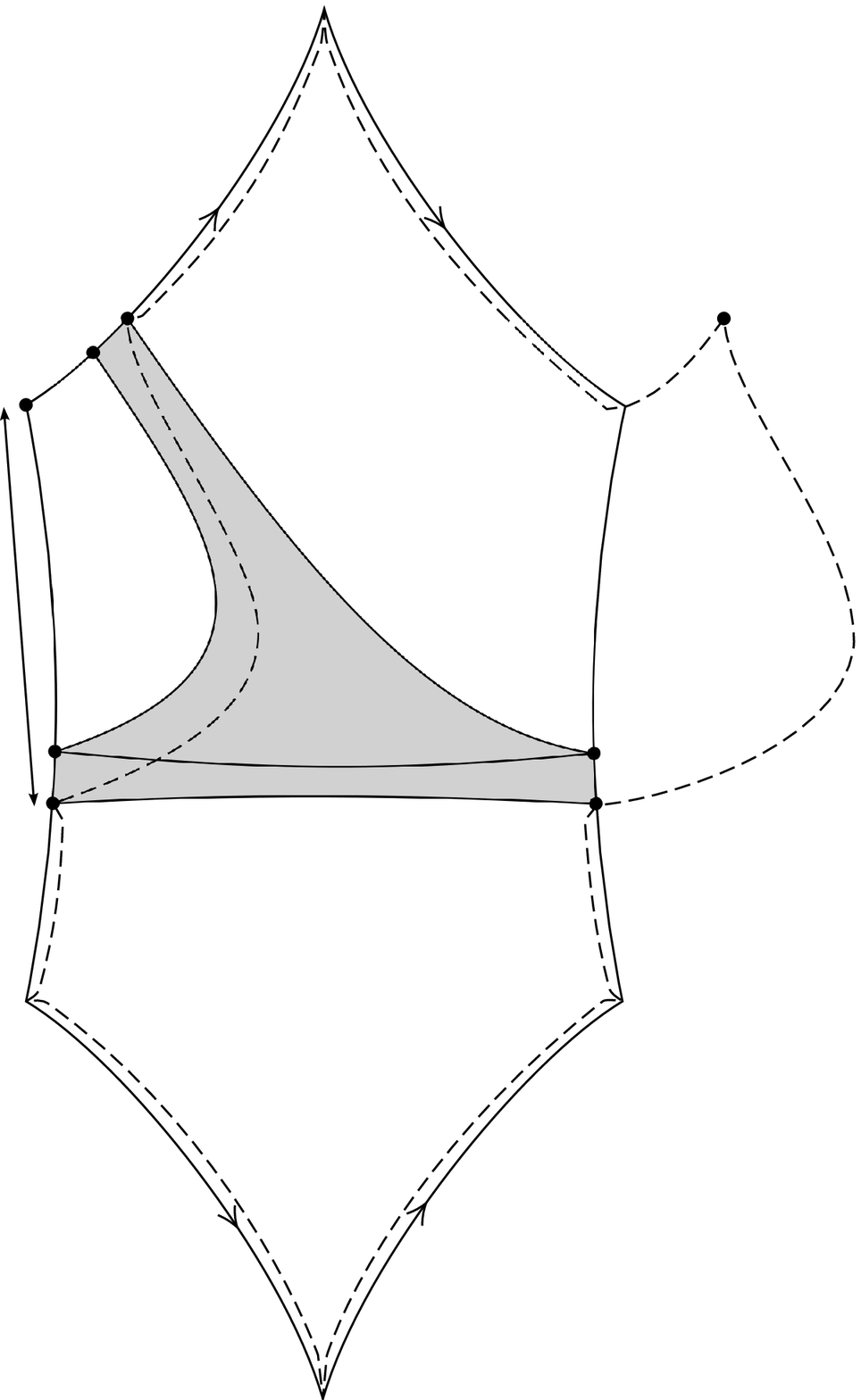} \hspace{1cm} \includegraphics[height=8cm]{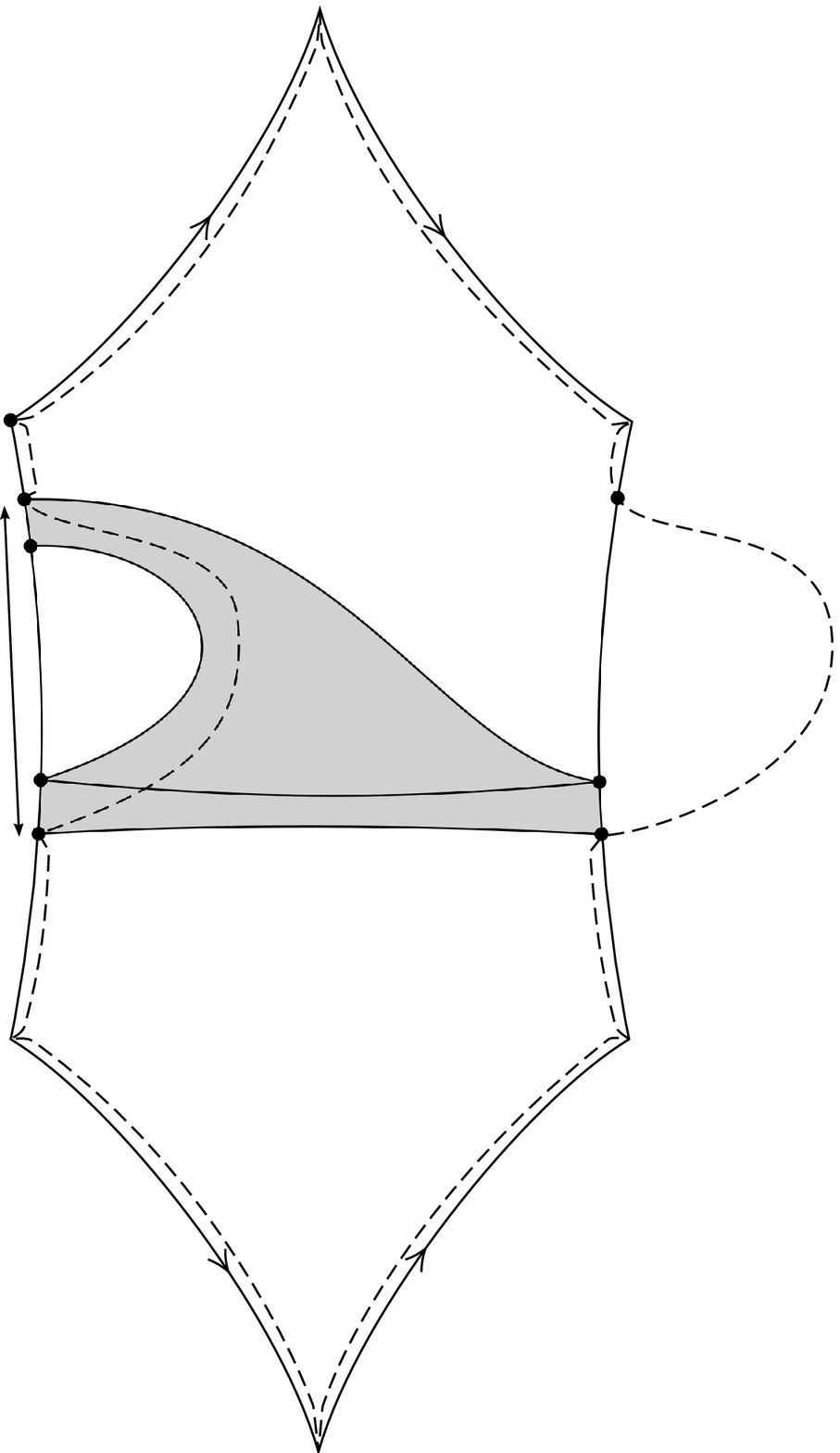}
		\caption{In both diagrams, since $d_S(x,w_{i-1})=1$, the dotted hexagon will give a \pcconj{} of shorter $S$--length. The shaded regions represent connected edges.}\label{fig:more short cuts}
	\end{figure}

	Since each of $s_i,g,h$ are all in some subgroup $H_\omega$, and the number of pairs of conjugate elements $g,h$ in $H_\omega$ of $X$--length bounded by $5D$ is finite, it follows that $s_i$ can be replaced by  an element that belongs to the finite set of minimal length conjugators between such elements. Hence, we choose $M$ large enough so that each of these elements has $X$--length at most $M$, and we are done.
\end{proof}

\begin{lem}\label{lem:long hexagon is skinny}
	Suppose that  $w$ is a \pcconj{} of minimal $S$--length such that {$\modulus{w}_S> 8\delta+B+1$}. Then for each $4\delta +B< i < n - 4\delta-B$,
	$$d_X(w_i,v'w_i) \leq 8\delta \max\{ 7 D, m\}$$
	where $B$ is the constant of Lemma \ref{lem:quasi-pcl on hyperbolic Cayley graph}, $D$ is the constant of Lemma \ref{lem:n-gon}, and $m$ is the constant of Lemma \ref{lem:nocancellation}.
	
	In particular, 
	$$\modulus{w}_S\leq  8\delta +B+1 + \abs{\ball_{X}(8\delta\max\{7D,m\})}.$$
\end{lem}

\begin{proof}
	Consider the pentagons obtained from the hexagon $\cQ$ by cutting along a geodesic $p$ from $w_i$ to $v'w_i$. By Lemma \ref{lem:quasi-pcl on hyperbolic Cayley graph}, $\ell(p)\leq 8\delta$. Let $e$ be an edge of $p$. Suppose that $e$ is an $H_\omega$--component and $\d_X(e_-,e_+)> \max\{7D, m\}$.
	By Lemma \ref{lem:n-gon}, $e$ is not isolated in either of the pentagons, and moreover if $e$ is connected to a component $f$ {in $\cQ$},
	then since $\d_X(e_-,e_+)>m$, $f$ must be an $H_\omega$--component as well. Note also that since $p$ is geodesic, $e$ is not connected to any other component of $p$. Therefore there is a component in the upper pentagon connected to a component in the lower pentagon. This either contradicts that $w$ is geodesic, $w$ is a \pcconj{} of minimal $S$--length, or that $|w|_S>1.$ This proves that $\d_X(e_-,e_+)\leq \max\{7D, m\}$ and hence
	$$d_X(w_i,v' w_i) \leq (8\delta) \max\{7D, m\}.$$

	Suppose now that $|w|_S>  8\delta +B+1 + \abs{\ball_{X}(8\delta\max\{7D,m\})}$. Then there exists $4\delta+B<i<j<|w|_S-4\delta-B$ such that $w_i^{-1}v'w_i= w_j^{-1}v'w_j$ and therefore $s_1\dots s_{i}s_{j+1}\dots s_n$ is a shorter \pcconj{} of $u$ and $v$, contradicting the minimality assumption.
	\end{proof}

We now deal with the case where none of $u$, $v$ or any of their cyclic permutations  is in a parabolic subgroup. The following Lemma tells us that we may, if necessary, change $w$ to another \pcconj{} which is still of minimal $S$--length, but for which, as a word on $S$, we have a control on the $X$--length of its first and last letters.

\begin{lem}\label{lem:bounding_chunks}
	Suppose that no  cyclic permutation of $u$ is contained in a parabolic subgroup. There exists a \pcconj{} $w'$ with $\modulus{w'}_S=\modulus{w}_S$ such that $w'$ is represented by a word $s'_1\ldots s'_n$ on $S$ with
			$$\abs{s'_n}_X\leq \max\{m,7D\}$$
	where $D$ is the constant of Lemma \ref{lem:n-gon} and $m$ is the constant of Lemma \ref{lem:nocancellation}.
	
	Similarly, if no cyclic permutation of $v$ is contained in a parabolic subgroup, then $w'$ can be chosen with
			$$\abs{s'_1}_X\leq \max\{m,7D\}.$$
\end{lem}

\begin{proof}	
We prove the bound on the length of the last letter. The bound on the first is symmetric. Let $e$ be the last edge in the path from $1$ to $w$.

	If $e$ is isolated in $\cQ$, then we use Lemma \ref{lem:n-gon}, applied to the $7$--gon obtained by treating this edge as a seventh side, and get the required result.
	
	If $e$ is not isolated in $\cQ$, and it is connected to an edge in the other $w$--side, then by Lemma \ref{lem:bands} it must be contained in a horizontal band. However this contradicts the assumption that no cyclic permutation of $u$ is contained in a parabolic subgroup.
	
	We may therefore assume that $e$ is not isolated and is not connected to the opposite $w$--side. It must therefore be connected to an edge in a side labelled by a subword of $u$.

	Assume that $u\equiv u_1u_2$ and $u'\equiv u_2u_1$. Then the path $p_2$ from 
	$w$ to $wu_2$ has label $\wh{u_2}$ and the path $p_1$ from $wu_2$ to $wu'$ has label $\wh{u_1}$. After cyclic permutation of $u$, we can assume that $e$ is connected to the first edge of $p_2$ (see Figure \ref{fig:cyclic chunk}~(a)) and without loss of generality we can assume further that  $w$ is chosen so that the length of {$p_1$ is minimal among}  all cyclic permutations of $u$ in which the corresponding hexagon $\cQ $ for some \pcconj{} of minimal $S$--length has a corner chunk adjacent to the vertex $w$.

	\begin{figure}[h!]
		\labellist
			\small \hair 2pt
			\pinlabel $\textrm{(a)}$ at 197 630
			\pinlabel $w$ at -10 455
			\pinlabel $p_2$ at 60 524
			\pinlabel $e$ at -7 425
			\pinlabel $w$ at -10 146
			\pinlabel $p_1$ at 185 228
			\pinlabel $e$ at -7 116
			\pinlabel $\textrm{(b)}$ at 700 630
			\pinlabel $w$ at 470 445
			\pinlabel $e$ at 470 408
			\pinlabel $f$ at 515 435
			\pinlabel {new $p_1$} at 930 580
		\endlabellist
		\vspace{15pt}
		\includegraphics[height=8cm]{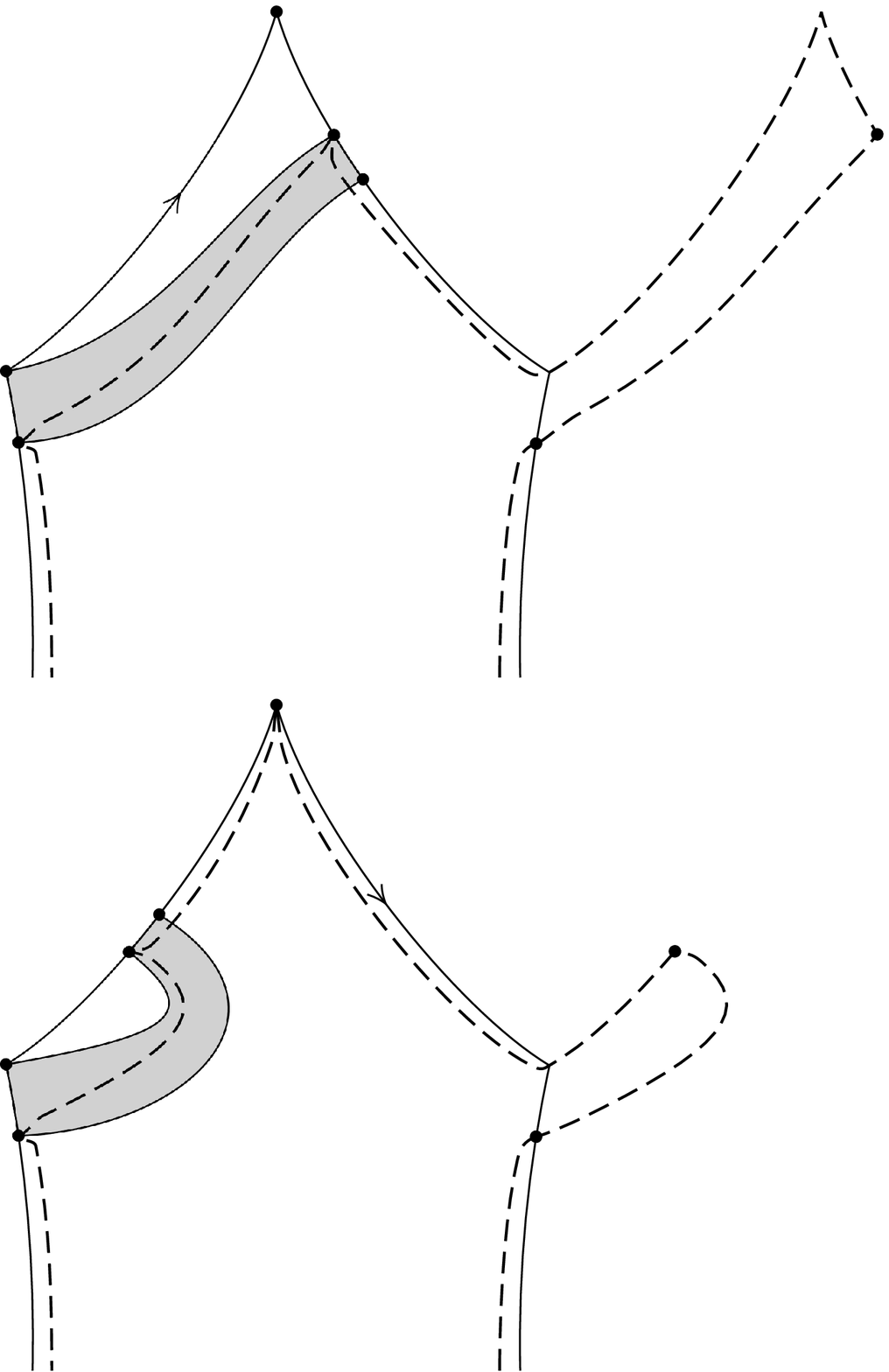}
		\hspace{25pt}
		\includegraphics[height=8cm]{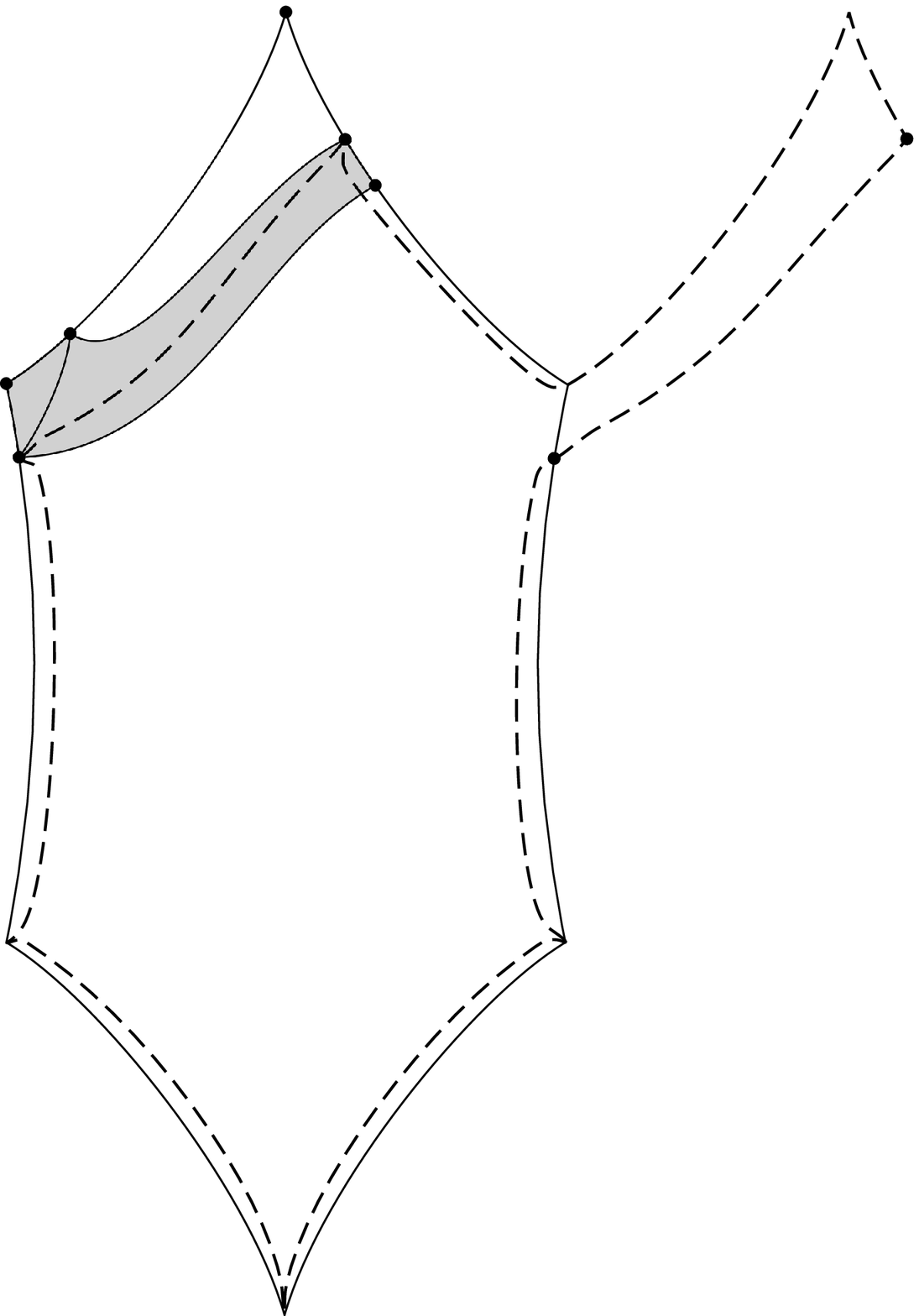}
		\caption{(a) Changing the last letter of the \pcconj{} and cyclically shifting $u$ enables us to assume that the edge $e$ is connected to the first edge of $p_2$. (b) If $f$ is connected to an edge in $p_1$ then we get a new hexagon realising a \pcconj{} for $u,v$ but where the side corresponding to $p_1$ is shorter.}\label{fig:cyclic chunk}
	\end{figure}
	
	Let $f$ be the edge of the interior of the corner chunk, see Figure \ref{fig:cyclic chunk} (b). If $f$ is isolated in the $7$--gon obtained by cutting out the corner chunk and treating this edge as a seventh side, then we get that $f$ has $X$--length bounded by $7D$ by Lemma \ref{lem:n-gon}. If $f$ is not isolated, the minimality of $p_1$ implies it cannot be connected to any edge of $p_1$ (see Figure \ref{fig:cyclic chunk}~(b)). Also, by Lemma \ref{lem:gen set lemma}~(4), if $f$ is connected to an edge of $p_2$ outside of the corner chunk, it must be the first such edge and again, by Lemma \ref{lem:gen set lemma}~(4) the two first edges of $p_2$ belong to different parabolic subgroups and therefore the $X$--length of $f$ is at most $m$ by Lemma \ref{lem:nocancellation}. 
	
	The Lemma follows by performing one cyclic permutation of $u$ and putting the label of $f$ as the last letter of $w$, leaving all other letters unchanged.  
\end{proof}

We have proved that for (quasi-)geodesic words $u$, $v$ that don't belong, up to cyclic permutation, to a parabolic subgroup, there is a \pcconj{} of bounded $X$--length. 

\begin{prop}\label{prop:PCL loxodromics}
Suppose that $u$ and $v$ are geodesic words such that no cyclic permutation of either is in a parabolic subgroup. Then
$$\PCL_{G,X}(u,v)\leq K$$ 
where $K$ is the constant $K=\big(8\delta+B+1+\abs{\ball_X(8\delta \max \{7D,m\})}\big) \max\{M,7D, m\}.$
\end{prop}
\begin{proof}
Let $w$ be a \pcconj{} of $u$ and $v$ of minimal $S$--length.
By Lemma \ref{lem:long hexagon is skinny}, 
$|w|_S\leq 8\delta+1+B+|\ball_X(8\delta \max\{7D,m\})|.$ 
By Lemmas \ref{lem:bounding_middle_conj} and \ref{lem:bounding_chunks} the $X$--length of the letters of $w$ is bounded by $\max\{M,7D,m\}$.
\end{proof}
We note that Proposition \ref{prop:PCL loxodromics}  generalises  \cite[Theorem 9.13]{AC14} and \cite[Theorem 3.14]{Bumagin} which only apply to 
the case when $u$ and $v$ are cyclic (quasi--)geodesics.

In the case when $u$ (or $v$) has a cyclic permutation in some parabolic $H_\omega$, we want to use the $\PCL$ for $H_\omega$. In order to do that, we need that $u$ is a word in $H_\omega \cap X$. The following Lemma states that, after paying a fixed price, we can assume this is the case.

\begin{lem}\label{lem:changing words to be fully parabolic}
Suppose $u$ has a cyclic permutation in $H_\omega$. 
Then there exists a word $U$ over $X\cap H_\omega$, such that $\abs{U}_X\leq \abs{u}_X$ and
  $$\PCL_{G,X}(u,v)\leq \PCL_{G,X}(U,v)+\max\{3D,m\}.$$ 
\end{lem}

\begin{proof}
Let $u_1,u_2$  be subwords of $u$ such that $u \equiv u_1 u_2$ and $h=u_2u_1\in H_\omega$. 
Consider the triangle $p_2p_1e^{-1}$ with an edge path $e$  labelled by $h$,
 and paths $p_1$ and $p_2$
labelled by $\wh{u_1}$ and $\wh{u_2}$ respectively. By Lemma \ref{lem:gen set lemma}~(4), $p_1,p_2$ are quasi-geodesic.

If $e$ is isolated in the triangle, then,  by Lemma \ref{lem:n-gon},
$\abs{h}_X\leq 3D.$ By Lemma \ref{lem:gen set lemma}~(2), there is a word $U$ over $X\cap H_\omega$, such that $U=h$ and $\abs{U}_{X\cap H_\omega}=\abs{U}_X=\abs{h}_X$. Note that $u_2u_1$ is conjugated to any cyclic permutation of $U$ by a subword of $U$. Hence $\PCL(u,v)=\PCL(U,v)+3D$.

If $e$ is not isolated in the triangle then it is connected to one of the other sides, 
say $p_2$.

Recall that by Lemma \ref{lem:gen set lemma}~(4), no subpath of $p_1$ and $p_2$ of length greater than one has the same endpoints as an edge with label in $\cH$. Thus $e$ has to  be connected to the first edge of $p_2$, which we denote by $f$. Let $g$ be the label of the edge joining the distinct vertices of $e$ and $f$. Let $p_2'$ be the subpath of $p_2$ obtained by removing $f$.
See Figure \ref{fig:triangle}.

If the side labelled $g$ is isolated in the triangle with other sides $p'_2$ and $p_1$, then $\abs{g}_X\leq 3D$ by Lemma \ref{lem:n-gon}. 
If it is not isolated, then 
it cannot be connected to any other $H_\omega$--component since this 
would contradict the facts that adjacent letters of $\wh{u}$  do not come from the same parabolic subgroups (so it can't be connected to the last edge of $p_1$) and that no subpath of $p_1$ or $p_2$ 
of length greater than $1$ has the same endpoints as an edge with label in $\cH$. In this case  $\abs{g}_X\leq m$ by Lemma \ref{lem:nocancellation}.

	\begin{figure}[h!]
		\labellist
			\small \hair 2pt
			\pinlabel $f$ at 18 70
			\pinlabel $e$ at 123 24
			\pinlabel $g$ at 140 80
			\pinlabel $p'_2$ at 80 146
			\pinlabel $p_1$ at 183 126
		\endlabellist
		\includegraphics[width=5cm]{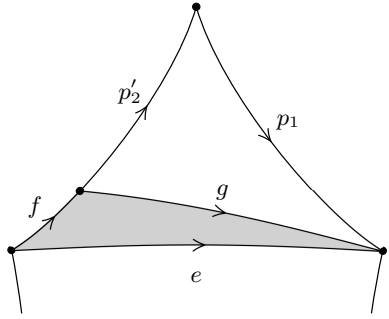}
		\caption{The component in the triangle $p_2p_1e^{-1}$ containing $e$.}\label{fig:triangle}
	\end{figure} 
	
To summarise, let $u_3$ be the prefix of $u_2$ which represents the label of the edge $f$. 
Note that $u_3$ is,  by definition,  a  geodesic word on $X\cap H_\omega$. 
Also by Lemma~\ref{lem:gen set lemma}~(2), there is a geodesic word $U_1$ on $X\cap H_\omega$ that is geodesic as a word over $X$ and represents $g$. 
We take  $U$ as the concatenation of $u_3$ and $U_1$. 
Any cyclic permutation of $U$ is conjugated by an element of $X$--length at most $\modulus{g}_X\leq\max\{3D,m\}$ to a cyclic permutation of $u$. The Lemma now follows.
\end{proof}

Thus we have sufficient information to deduce the main theorem.

\begin{proof}[Proof of Theorem 1]

We take $K$ as in Proposition \ref{prop:PCL loxodromics}. 
Let $u$ and $v$ be two geodesic words in $X$.
We can assume that $\modulus{u}_X,\modulus{v}_X\geq \max\{5D,m\}$ by increasing $K$, if necessary, so that
$\PCL_{G,X}(\max\{10D,2m\})\leq K$.
Suppose that no cyclic permutation of either $u$ or $v$ labels an element of a parabolic subgroup. Then, by Proposition \ref{prop:PCL loxodromics}, $ \PCL_{G,X}(u,v)\leq K.$

We have to consider the case when $u$ or $v$ has a cyclic permutation in a parabolic. By symmetry, it is enough to consider the case when $u$ can be cyclically permuted into a parabolic $H_\omega$. By  Lemma \ref{lem:changing words to be fully parabolic}, and increasing $K$ if necessary, we can assume that $u$ is in fact a word over $X\cap H_\omega$.

Let  $w\equiv s_1\dots s_n$ be a \pcconj{} of minimal $S$--length.  By Lemmas \ref{lem:bounding_middle_conj} and \ref{lem:long hexagon is skinny}, $n\leq 8\delta +B +1 +|\ball_X(8\delta\max\{7D,m\})|$ and $\modulus{s_i}_X\leq M$ for $i=2,\dots, n-1$. So we only need to bound $\modulus{s_i}_X$ for $i=1,n$.

In this situation, instead of the hexagon $\cQ$, we can consider the  (possibly degenerate) pentagon $\cQ'$ in $\ga$, where the sides from $w$ to $wu_2$ and from $wu_2$ to $wu'$ are substituted by an edge from $w$ to $wu'$ (see Figure \ref{fig:pentagon}(a)). Call this edge $e$.
	
	\begin{figure}[h!]
		\labellist
			\small \hair 2pt
			\pinlabel $1$ [r] at 4 163
			\pinlabel $w$ [r] at 4 403
			\pinlabel $wu'=v'w$ [b] at 244 404
			\pinlabel $v'$ [l] at 246 165
			\pinlabel $v_2$ [t] at 124 1
			\pinlabel $e$ [b] at 125 403
			\pinlabel $e$ [b] at 445 403
			\pinlabel $e'$ [t] at 445 370
			\pinlabel $w_{n-1}$ [r] at 333 366
			\pinlabel $v'w_{n-1}$ [l] at 562 366
			\pinlabel $\textrm{(a)}$ at 122.5 454
			\pinlabel $\textrm{(b)}$ at 445 454
			\endlabellist
			\vspace{15pt}
		\includegraphics[height=5cm]{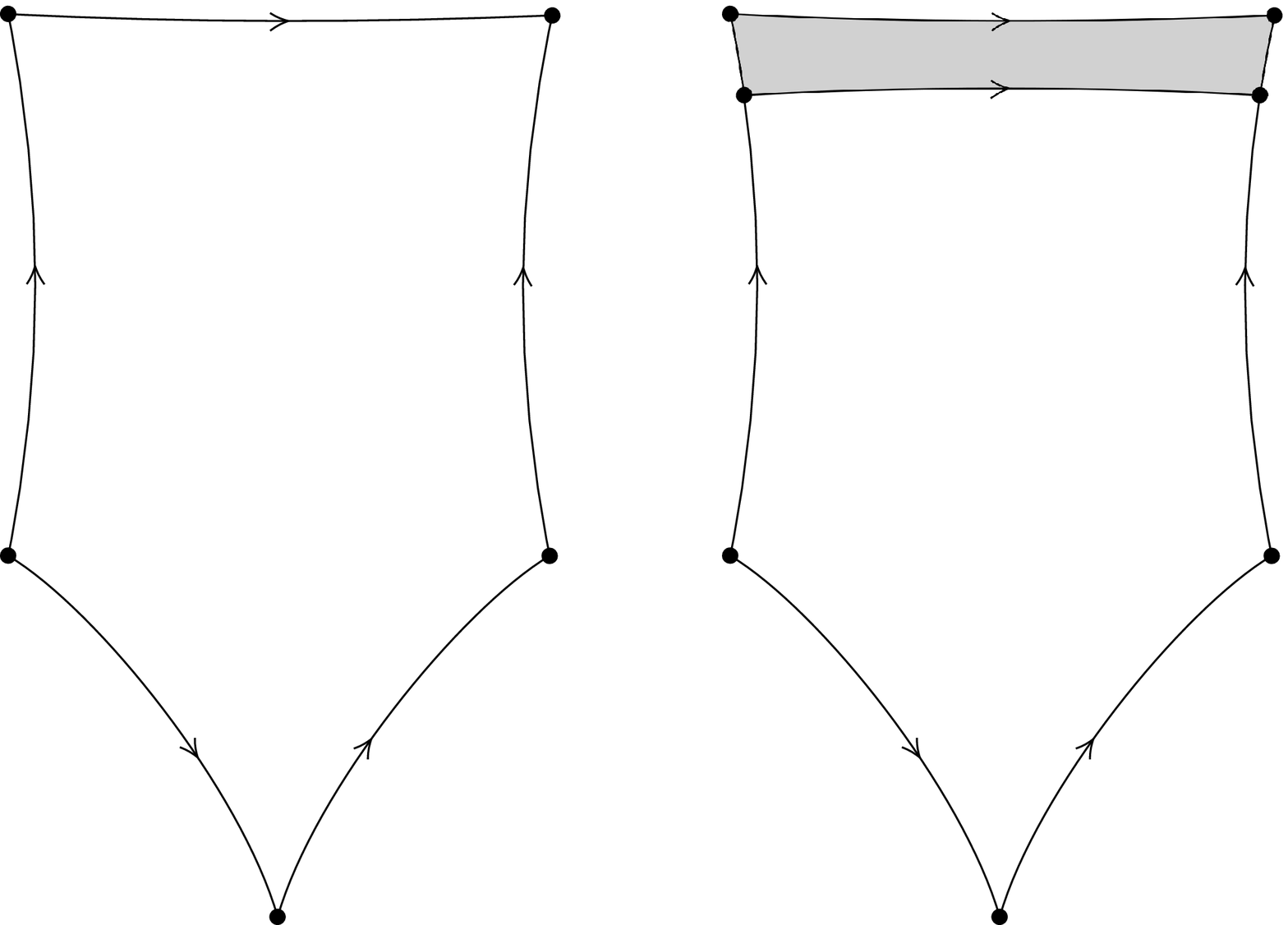}
		\caption{The pentagon $\cQ'$.}\label{fig:pentagon}
	\end{figure}
{\bf Case 1.} Suppose that $v$ is not, up to cyclic permutation in a parabolic.

By Lemma \ref{lem:bounding_chunks}, {we may assume that} $\modulus{s_1}_X\leq \max\{m,7D\}$, hence we need to bound $\modulus{s_n}_X$ in the case that $n>1$.

If the $s_n$--edge of the path from $1$ to $w$ is isolated, then $\modulus{s_n}_X\leq 6D$. If this is not the case, the argument of Figure \ref{fig:more short cuts} shows that it has to be either connected to the edge $e$ or to the $s_n$--edge in the opposite side. In both cases, we get a horizontal band. Let $e'$ denote the lower edge of the band (see Figure \ref{fig:pentagon}(b)).
Let $\cP$ be the ``subpentagon'' of $\cQ'$ obtained by deleting the band.
We remark that  if  $e'$ is not isolated and is of $X$-length greater than $m$, we would contradict the minimality of $\modulus{w}_S$. Therefore, by Lemma \ref{lem:n-gon}, $|e'|_X\leq \max\{5D,m\}$. Hence we may assume that $$|s_n|_X\leq \PCL_{H_\omega, X\cap H_\omega}(|u|_X+5D)+\max\{5D,m\}.$$ Since we are assuming that $|v|_X> \max\{5D,m\}$,
we have that $$|s_n|_X\leq \PCL_{H_\omega, X\cap H_\omega}(n)+ \max\{5D,m\}.$$
The result now follows in this case.

{\bf Case 2.} Suppose now that some cyclic permutation of $v$ is in a parabolic $H_\mu$. Again, by Lemma \ref{lem:changing words to be fully parabolic}, and increasing $K$ if necessary, we can assume that $v$ is a word in $H_\mu \cap X$.

If $n>1$, the previous case shows that we can assume that $$\modulus{s_n}_X\leq \PCL_{H_\omega, X\cap H_\omega}(|u|_X+\max\{5D,m\})+\max\{5D,m\}$$ and  $$\modulus{s_1}_X\leq \PCL_{H_\omega, X\cap H_\omega}(|v|_X+\max\{5D,m\})+\max\{5D,m\}.$$

Finally,  suppose $n=1$. 
If the side from $1$ to $w$ is isolated in $\cQ '$ (which is now a quadrilateral),  
then $\modulus{w}_X\leq 4D$. 
If it is not, then $H_\omega=H_\mu$ and $w\in H_\omega$. {In this case, we can take $w$ so that} $|w|_X\leq \PCL_{H_\omega,X\cap H_\omega}(|u|_X+|v|_X)$.
\end{proof}


\begin{thebibliography}{1}
	
	\bibitem{AC14}
	Y.\ Antol\'in and L.\ Ciobanu,
	\newblock {\em Finite generating sets of relatively hyperbolic groups and applications to geodesic languages.}
	\newblock Trans. Amer. Math. Soc. (to appear).
	
	
	\bibitem{BuckleyHolt}
	D.\ J.\ Buckley and D.\ F.\ Holt
	\newblock {\em The conjugacy problem in hyperbolic groups for finite lists of group elements.}
	\newblock Internat. J. Algebra Comput. {\bf 23} (2013), no. 5, 1127–1150.
	
	
	\bibitem{BH}
	M.\ Bridson and A.\ Haefliger,
	\newblock{\em Metric spaces of non-positive curvature.}
	\newblock  Grundlehren der
	Mathematischen Wissenschaften [Fundamental Principles of Mathematical Sciences], 319.
	Springer-Verlag, Berlin, 1999. xxii+643 pp.
	
	\bibitem{BridsonHowie}
	M.\ Bridson and J.\ Howie, 
	\newblock{\em Conjugacy of finite subsets in hyperbolic groups}.
	\newblock Internat. J. Algebra Comput. {\bf 15} (2005), 725--756.
	
	\bibitem{Bowditch}
	B.\ H.\ Bowditch,
	\newblock{\em Relatively hyperbolic groups}.
	\newblock  Internat. J. Algebra Comput. 22 (2012), no. 3, 1250016, 66 pp. 
	
	\bibitem{Bumagin}
	I.\ Bumagin, 
	\newblock{\em The conjugacy problem for relatively hyperbolic groups}.
	\newblock Algebraic \& Geometric Topology. Vol. 4, (2004) 1013--1040.
	
	\bibitem{Bumagin14}
	I.\ Bumagin, 
	\newblock{\em Time complexity of the conjugacy problem in relatively hyperbolic groups}.
	\newblock Preprint arXiv: 1407.4528.
	
	\bibitem{delaHarpe}
	P. de la Harpe, 
	\newblock{\em Topics in geometric group theory.}
	\newblock Chicago Lectures in Mathematics. University of Chicago Press, Chicago, IL, 2000
	
	\bibitem{CH}
	L.\ Ciobanu and S.\ Hermiller,
	\newblock {\em Conjugacy growth series and languages in groups}.
	\newblock Trans. Amer. Math. Soc., 366 (2014), 2803--2825. 
	%
	\bibitem{CHHR}
	L.\ Ciobanu, S.\ Hermiller, D.\ F.\ Holt and S.\ Rees,
	\newblock {\em Conjugacy Languages in Groups},
	\newblock Preprint arXiv:1401.7203, Israel Journal of Mathematics, to appear.

	%
		\bibitem{Dahmani}
		F.\ Dahmani,
		\newblock{\em Combination of  convergence  groups.}
		\newblock Geom. Topol. {\bf 7} (2003) 933--963. 

	%
	\bibitem{EpsteinHolt}
	D.\ B.\ A.\ Epstein and D.\ Holt,
	\newblock{\it The linearity of the conjugacy problem in word hyperbolic groups.}
	\newblock  Internat. J. Algebra Comput. {\bf 16} (2006), 287--305.
	
	\bibitem{Farb}
	B.\ Farb,
	\newblock {\it Relatively hyperbolic groups.}
	\newblock  Geom. Funct. Anal. {\bf 8} (1998), no. 5, 810--840. 

	
	\bibitem{GubaSapir}
	V.\	Guba and M.\ Sapir, 
	\newblock {\it On the conjugacy growth functions of groups.} 
	\newblock Illinois Journal of Mathematics, 54(1) (2010), 301--313.

	
	
	\bibitem{JiOgleRamsey}
	R.\ Ji, C.\ Ogle and B.\ Ramsey
	\newblock {\em Relatively hyperbolic groups, rapid decay algebras, and a generalization of the Bass conjecture.}
	\newblock  J. Noncommut. Geom. {\bf 4} (2010), no. 1, 83–124. 
	
	

	\bibitem{BFC}
	B.\ H.\ Neumann, 
	\newblock {\em Groups with finite classes of conjugate subgroups.}
	\newblock Math. Z. 63 (1955), 76--96. 
	

	
	\bibitem{Zoe}
	Z.\ O'Connor,
	\newblock {\it Conjugacy Search Problem for Relatively Hyperbolic Groups}.
	\newblock Preprint arXiv:1211.5561.
	

	
	\bibitem{Osin06}
	D.\ Osin,
	\newblock {\it Relatively hyperbolic groups: intrinsic geometry, algebraic properties, and 
		algorithmic problems.}
	\newblock Mem. Amer. Math. Soc. {\bf 179} (2006), no. 843, vi+100 pp. 
	
	\bibitem{OsinPF}
	D.\ Osin,
	\newblock {\em Peripheral fillings of relatively hyperbolic groups.}
	\newblock Invent. Math. \textbf{167 }(2007), no. 2,
	295--326.
	
	
	
	
	
	
	\bibitem{Saleext}
	A.\ W.\ Sale,
	\newblock {\em Conjugacy length in group extensions.}
	\newblock Comm. Algebra (to appear).
	
	\bibitem{Salewreath}
	A.\ W.\ Sale,
	\newblock {\em The geometry of the conjugacy problem in wreath products and free solvable groups.}
	\newblock J. Group Theory (to appear).
	
	\bibitem{Salelie}
	A.\ W.\ Sale,
	\newblock {\em Bounded conjugators for real hyperbolic and unipotent elements in semisimple Lie groups.}
	\newblock J.Lie Theory, {\bf 24} (2014), no. 1, 259-305.
\end{thebibliography}
\end{document}